\documentclass[10pt]{amsart}
\usepackage{eepic}
\usepackage{epsfig}
\usepackage{graphicx}
\usepackage{eufrak}
\usepackage{mathrsfs}
\usepackage[english]{babel}
\usepackage{color}
\usepackage{mathabx}

\usepackage{palatino, mathpazo}
\usepackage{amscd}      
\usepackage{amssymb}
\usepackage{amsmath}
\usepackage{dsfont}
\usepackage{xypic}      
\LaTeXdiagrams          
\usepackage[all,v2]{xy}
\xyoption{2cell} \UseAllTwocells \xyoption{frame} \CompileMatrices
\usepackage{latexsym}
\usepackage{epsfig}
\usepackage{amsfonts}
\usepackage{enumerate}

\allowdisplaybreaks[3]

\newtheorem{prop}{Proposition}[section]
\newtheorem{lem}[prop]{Lemma}

\newtheorem{thm}[prop]{Theorem}
\newtheorem{rmk}[prop]{Remark}

\newtheorem{defn}[prop]{Definition}

            {\nolinebreak $\Box$ \end{trivlist}}

\newcommand{\noprint}[1]{}

\newcommand{\Ext}{\mbox{Ext}}

\newcommand{\Hom}{\mbox{Hom}}

\newcommand{\pt}{\mathop{pt}}

\newcommand{\E}{\mathop{\sf E}\nolimits}

\newcommand{\N}{\mathcal{N}}

\newcommand{\Tt}{{\mathfrak t}}

\newcommand{\zz}{{\mathbb Z}}
\newcommand{\hh}{{\mathbb H}}

\renewcommand{\ll}{{\mathbb L}}
\newcommand{\qq}{{\mathbb Q}}
\newcommand{\pp}{{\mathbb P}}
\newcommand{\cc}{{\mathbb C}}

\newcommand{\sC}{{\mathcal C}}
\newcommand{\sE}{{\mathcal E}}
\newcommand{\sI}{{\mathcal I}}

\newcommand{\sL}{{\mathcal L}}

\newcommand{\sS}{{\mathcal S}}

\newcommand{\sO}{{\mathcal O}}

\newcommand{\sX}{{\mathcal X}}

\newcommand{\sZ}{{\mathcal Z}}

\newcommand{\sF}{{\mathcal F}}

\newcommand{\Coh}{\mbox{Coh}}

\newcommand{\rE}{\mathscr{E}}
\newcommand{\rM}{\mathscr{M}}

\newcommand{\cHom}{\mathscr{H}om}

\DeclareMathOperator{\loc}{loc}

\DeclareMathOperator{\id}{id}

\DeclareMathOperator{\Hilb}{Hilb}

\DeclareMathOperator{\Aut}{Aut}

\DeclareMathOperator{\VW}{VW}
\DeclareMathOperator{\vw}{vw}

\DeclareMathOperator{\vir}{vir}
\DeclareMathOperator{\mov}{mov}
\DeclareMathOperator{\Pic}{Pic}
\DeclareMathOperator{\vd}{vd}
\DeclareMathOperator{\Ch}{Ch}
\DeclareMathOperator{\CR}{CR}
\DeclareMathOperator{\PD}{PD}
\DeclareMathOperator{\AJ}{AJ}

\DeclareMathOperator{\Higg}{Higg}

\DeclareMathOperator{\vb}{vb}
\DeclareMathOperator{\odd}{odd}
\DeclareMathOperator{\even}{even}

\newcommand{\rk}{\mathop{\rm rk}}

\newcommand{\ev}{\mathop{\rm ev}\nolimits}

\newcommand{\tr}{\mathop{\rm tr}\nolimits}

\renewcommand{\Im}{\mathop{\rm Im}}
\renewcommand{\top}{\mathop{\rm top}}

\newcommand{\rank}{\mathop{\rm rank}\nolimits}

\newcommand{\Jac}{\mathop{\rm Jac}\nolimits}

\newcommand{\proj}{\mathop{\rm Proj}\nolimits}

\numberwithin{equation}{subsection}

\newcommand {\mat}      [1] {\left(\begin{array}{#1}}
\newcommand {\rix}          {\end{array}\right)}

\title[VW invariants and S-duality]{The Vafa-Witten invariants via surface Deligne-Mumford stacks and S-duality}
\author{Yunfeng Jiang}
\address{Department of Mathematics\\ University of Kansas\\ 405 Snow Hall 1460 Jayhawk Blvd\\Lawrence KS 66045 USA} 
\email{y.jiang@ku.edu}

\begin{document}
\sloppy \maketitle
\begin{abstract}
Motivated by the S-duality conjecture of Vafa-Witten, Tanaka-Thomas have developed a theory of Vafa-Witten invariants for projective surfaces using the moduli space of Higgs sheaves.
Their definition and calculation prove the S-duality prediction of Vafa-Witten  in many cases in the side of gauge group $SU(r)$.  In this survey paper for ICCM-2019 we review the S-duality conjecture in physics by Vafa-Witten  and the definition of Vafa-Witten invariants for smooth projective surfaces and surface Deligne-Mumford stacks. We make a prediction that the Vafa-Witten invariants for Deligne-Mumford surfaces may give the generating series for the  Langlands dual group $^{L}SU(r)=SU(r)/\zz_r$. We survey a check  for the projective plane $\pp^2$. 
\end{abstract}

\maketitle

\tableofcontents

\section{Introduction}

In this paper we survey some results for the  Tanaka-Thomas's Vafa-Witten invariants for projective surfaces \cite{TT1}, \cite{TT2} and  two dimensional smooth Deligne-Mumford (DM) stacks (called surface DM stacks)  in \cite{JP}.   We provide evidences that the Vafa-Witten invariants of surface DM stacks may give  candidates for the invariants in algebraic geometry for the Langlands dual group $^{L}SU(r)=SU(r)/\zz_{r}$.

\subsection{History of S-duality}

The  motivation for the Vafa-Witten invariants from physics is the S-duality conjecture of $N =4$ supersymmetric Yang-Mills theory on a real 4-manifold $M$ \cite{VW}, where by physical duality theory  Vafa and Witten  \cite{VW} predicted that the generating (partition) function of the invariants of the moduli space of instantons  on projective surfaces should be modular forms.  

This theory involves coupling constants $\theta, g$ combined as follows
$$\tau:= \frac{\theta}{2\pi} + \frac{4\pi i}{g^2}.$$
The S-duality predicts that the transformation $\tau\to -\frac{1}{\tau}$ maps the partition function for gauge group $G$ to the partition function with Langlands dual gauge group $^{L}G$.  Vafa-Witten consider a $4$-manifold $M$ underlying a smooth projective surface $S$ over $\cc$ and $G =SU(r)$. The Langlands dual group 
$^{L}SU(r)=SU(r)/\zz_r$. We make these transformations more precise following \cite[\S 3]{VW}. 
Think $\tau$ as the parameter of the upper half plane $\hh$.  Let $\Gamma_0(4)\subset SL(2,\zz)$ be the subgroup
$$\Gamma_0(4)=\left\{
\mat{cc} a&b\\
c&d\rix\in SL(2,\zz):  4| c\right\}.$$
The group $\Gamma_0(4)$ acts on $\hh$ by
$$\tau\mapsto \frac{a\tau+b}{c\tau+d}.$$ The group $SL(2,\zz)$ is generated by transformations
$$S=\mat{cc} 0&-1\\
1&0\rix;  \quad T=\mat{cc} 1&1\\
0&1\rix.$$
From \cite{VW}, invariance under $T$ is the assertion that physics is periodic in $\theta$ with period $2\pi$, and 
$S$
is equivalent at $\theta=0$ to the transformation $\frac{g^2}{4\pi}\mapsto (\frac{g^2}{4\pi})^{-1}$ originally proposed by
Montonen and Olive \cite{MO}. One can check  that 
$T(\tau)=\tau+1$, and $S(\tau)=-\frac{1}{\tau}$. 

For a smooth projective surface $S$,  let $Z(\tau, SU(r))=Z(q, SU(r))$ be the partition function which counts the invariants of instanton moduli spaces, where $q=e^{2\pi i \tau}$.  We let $Z(\tau, SU(r)/\zz_r)$ be the  partition function which counts the invariants of $SU(r)/\zz_r$-instanton moduli spaces.  As pointed out in \cite[\S 3]{VW}, when some vanishing theorem holds, the invariants count the Euler characteristic of  instanton moduli spaces. We will see a mathematical meaning of this vanishing. 

Then the S-duality conjecture of Vafa-Witten can be stated as follows:  the transformation $T$ acts on $Z(q, SU(r))$, and the $S$-transformation sends
\begin{equation}\label{eqn_S_transformation}
Z\left(-\frac{1}{\tau}, SU(r)\right)=\pm r^{-\frac{\chi}{2}}\left(\frac{\tau}{i}\right)^{\frac{\omega}{2}}Z(\tau, SU(r)/\zz_r).
\end{equation}
for some $\omega$, where  $\chi:=\chi(S)$ is the topological Euler number of $S$.  Usually $\omega=-\chi$.
This is Formula (3.18) in \cite{VW}. 
In mathematics we think $Z(\tau, SU(r))=Z(q, SU(r))$ as the partition function which counts the invariants of  moduli space of vector bundles or Higgs bundles  on $S$.   Let 
$$\eta(q)=q^{\frac{1}{24}}\prod_{k\geq 1}(1-q^k)$$
be the Dedekind eta function.  Let 
$$\widehat{Z}(\tau, SU(r))=\eta(q)^{-w}\cdot Z(q, SU(r)),$$
and we will see that $\widehat{Z}(\tau, SU(r))$ is the partition function of the moduli space of Gieseker sable sheaves or stable Higgs sheaves.
Then S-duality predicts:
\begin{equation}\label{eqn_S_transformation_2}
\widehat{Z}\left(-\frac{1}{\tau}, SU(r)\right)=\pm r^{-\frac{-\chi}{2}}\widehat{Z}(\tau, SU(r)/\zz_r)
\end{equation}
which is coming from Formula (3.15) in \cite{VW}. 
The transformation  $T^4$ acts on the $SU(r)/\zz_r$-theory to itself; and 
$ST^4 S=\mat{cc} 1&0\\
4&1\rix$ will map the $SU(r)$-theory to itself.   Note that 
$\Gamma_0(4)=\langle T, ST^4 S\rangle$ is generated by $T, ST^4 S$. 
 In the case of a spin manifold, we get the subgroup of $SL(2,\zz)$ generated by $S$ and $ST^2S$, which is the group 
$$\Gamma_0(2)=\left\{
\mat{cc} a&b\\
c&d\rix\in SL(2,\zz):  2| c\right\}.$$  
Therefore if the S-duality conjecture holds, the partition $Z(\tau, SU(r))=Z(q, SU(r))$ is a modular form with modular group 
$\Gamma_0(4)$ or $\Gamma_0(2)$ if $M=S$ is a spin manifold.

In \cite[\S 4]{VW}, Vafa-Witten checked the S-duality for the cases $K3$ surface and $\pp^2$, and gave a formula  (5.38) of \cite[\S 5]{VW}) for general type surfaces.  For $\pp^2$, Vafa-Witten used the mathematical results of Klyachko and Yoshioka, and for $K3$ surfaces, they predicted the formula from physics.  

In mathematics side people have studied the invariants for a long time using Donaldson invariants.  In algebraic geometry the invariants are the Euler characteristic of the moduli space of Gieseker or slope stable coherent sheaves on $S$. This corresponds to the case in \cite{VW} such that the obstruction sheaves  vanish.     The blow-up formula of the Vafa-Witten invariants in this case was proved by Li-Qin in \cite{LQ}. 

But to the author's knowledge there are few theories or defining invariants in algebraic geometry for the Langlands dual group $SU(r)/\zz_r$.  
In the rank $2$ case, the Langlands dual group $SU(2)/\zz_2=SO(3)$. 
There exists some theory for the $SO(3)$-Donaldson invariants for the surface $S$, see \cite{KM}, \cite{Gottsche}, \cite{MW}.

\subsection{Vafa-Witten invariants by Tanaka-Thomas}
In differential geometry solutions of the Vafa-Witten equation on a projective surface $S$ are given by 
polystable Higgs bundles on the surface $S$, see \cite{TT1}. 
The moduli space of Higgs bundles 
has a partial compactification by Gieseker semistable Higgs pairs $(E,\phi)$ on $S$, where $E$ is a torsion free  coherent sheaf with rank $\rk >0$, and $\phi\in \Hom_{S}(E, E\otimes K_S)$ is a section called a Higgs field.  

In \cite{TT1}, \cite{TT2} Tanaka and Thomas defined the Vafa-Witten invariants using the moduli space $\N$ of Gieseker semi-stable Higgs pairs $(E,\phi)$ on $S$ with topological data $(\rk=\rank(E), c_1(E), c_2(E))$.  We use $\rk$ to represent the rank of the torsion free sheaf $E$, and it will be the same as $r$ when studying the invariants for the Langlands group 
$SU(r)/\zz_r$.  By spectral theory, the moduli space $\N$ of Gieseker semi-stable Higgs pairs $(E,\phi)$ on $S$ is isomorphic to the moduli space of Gieseker semi-stable torsion sheaves  
$\sE_\phi$ on the total space $X:=\mbox{Tot}(K_S)$.  If the semistability and stability coincide,  the moduli space $\N$ admits a symmetric obstruction theory in \cite{Behrend} since $X$ is a smooth Calabi-Yau threefold.  There exists a dimension zero virtual fundamental cycle 
$[\N]^{\vir}\in H_0(\N)$.  The moduli space $\N$ is not compact, but it admits a $\cc^*$-action induced by the 
$\cc^*$-action on $X$ by scaling the fibres of $X\to S$.  The $\cc^*$-fixed locus $\N^{\cc^*}$ is compact,  then from \cite{GP}, $\N^{\cc^*}$ inherits a perfect obstruction theory from $\N$ and one can define invariants using virtual localization.  

But the obstruction sheaf in this case implies that this invariant  is zero unless 
$H^{0,1}(S)=H^{0,2}(S)=0$, and the reason is that  the obstruction sheaf contains a trivial summand. 
The right invariants are defined by using the moduli space $\N_L^{\perp}$ of Higgs pairs with fixed determinant $L\in\Pic(S)$ and trace-free  $\phi$. 
Tanaka-Thomas have carefully studied the deformation and obstruction theory of the Higgs pairs instead of using the deformation and obstruction theory for the corresponding torsion sheaves, and constructed a  symmetric obstruction theory on $\N^{\perp}_{L}$.  
The space $\N^{\perp}_{L}$ still admits a $\cc^*$-action, therefore inherits a perfect obstruction theory on the fixed locus. 
The Vafa-Witten invariants  are defined as:
\begin{equation}\label{eqn_localized_SU_invariants_intr}
\VW(S):=\VW_{\rk, c_1, c_2}(S)=\int_{[(\N_L^{\perp})^{\cc^*}]^{\vir}}\frac{1}{e(N^{\vir})}.
\end{equation}
This corresponds to the $SU(\rk)$ gauge group.   Tanaka-Thomas did explicit calculations for some surfaces of general type in \cite[\S 8]{TT1} and verified some part of Formula (5.38) in \cite{VW}.  Since for such general type surfaces, the 
$\cc^*$-fixed loci contain components such that the Higgs fields are non-zero,  there are really contributions from the threefolds to the Vafa-Witten invariants.  This is the first time that the threefold contributions are made for the Vafa-Witten invariants. 
Some calculations and the refined version of the Vafa-Witten invariants have been studied in \cite{Thomas2}, \cite{MT}, \cite{GK}, \cite{Laarakker}.  It is worth mentioning that in \cite{TT2}, using the definition of Vafa-Witten invariants and also the semistable ones defined  by Joyce-Song pairs, Tanaka-Thomas calculate and prove the prediction of Vafa-Witten in \cite[\S 4]{VW} for the K3 surfaces for the gauge group $SU(r)$. 

We introduce another invariants using Behrend functions. 
In \cite{Behrend},  Behrend defined an integer valued constructible function 
$\nu_{\N}: \N^{\perp}_{L}\to \zz$
called the Behrend function.  We can define 
\begin{equation}\label{eqn_localized_Behrend_invariants_intr}
\vw(S):=\vw_{\rk, c_1, c_2}(S)=\chi(\N_L^{\perp}, \nu_{\N})
\end{equation}
where $\chi(\N_L^{\perp}, \nu_{\N})$ is the weighted Euler characteristic.  The $\cc^*$-action on $\N^{\perp}_{L}$ induces a cosection $\sigma: \Omega_{\N^{\perp}_{L}}\to \sO_{\N^{\perp}_{L}}$ in \cite{KL} by taking the dual of the associated vector field $v$ given by the $\cc^*$-action. The degenerate locus is the fixed locus 
$(\N^{\perp}_{L})^{\cc^*}$, therefore there exists a cosection localized virtual cycle 
$[\N^{\perp}_{L}]^{\vir}_{\loc}\in H_0((\N^{\perp}_{L})^{\cc^*})$, and 
$$\int_{[\N^{\perp}_{L}]^{\vir}_{\loc}}1=\chi(\N_L^{\perp}, \nu_{\N})$$
as proved in \cite{JT}, \cite{Jiang}. 
Tanaka-Thomas proved that in the case $\deg K_{S}< 0$ and the case that $S$ is a K3 surface, 
$\VW(S)=\vw(S)$. They also prove their corresponding  generalized Vafa-Witten invariants in \cite{JS} also agree, see \cite{TT2} for the Fano case and  \cite{MT} for the K3 surface case.  
Using the weighted Euler characteristic in \cite{TT2} Tanaka-Thomas calculated the generating series of the Vafa-Witten invariants for K3 surfaces.  The formula (5.25) in \cite{TT2}  matches the prediction in \cite[\S 4]{VW}.

\subsection{Vafa-Witten invariants for surface DM stacks}

In \cite{JP}, we generalized Tanaka-Thomas's theory to surface DM stacks. The motivation to study the Vafa-Witten invariants of Deligne-Mumford surfaces is the S-duality conjecture. 
\'Etale gerbes on $S$ are interesting DM stacks.  We propose that the Vafa-Witten invariants for some \'etale gerbes on $S$ will give the mathematical invariants for the Langlands dual group $SU(r)/\zz_r$. 
The proof for $\pp^2$ and $K3$ surfaces in rank two will be given in \cite{Jiang_twist}.  It is also very interesting to study 
Vafa-Witten invariants for other DM surfaces. For instance the global quotient orbifold K3 surface $[K3/G]$ provides interesting testing examples. 

Let $\sS$ be a smooth two dimensional DM stack, which we call it  a surface DM stack.  
Let $p: \sS\to S$ be the map to its coarse moduli space and  fix a polarization $\sO_S(1)$. 
Choose a generating sheaf $\Xi$ which is a locally free sheaf and relatively $p$-very ample, and for a coherent sheaf $E$ on $\sS$, the modified Hilbert polynomial is defined by:
$$H_{\Xi}(E,m) = \chi(\sS,E\otimes \Xi^{\vee}\otimes p^*\sO_S(m)).$$
Then we can write down 
$$H_{\Xi}(E,m) =\sum_{i=0}^2\alpha_{\Xi, i}\frac{m^i}{i!}.$$
The reduced Hilbert polynomial for pure sheaves, which is denoted  by $h_{\Xi}(E)$;  is the monic polynomial with rational coefficients 
$\frac{H_{\Xi}(E)}{\alpha_{\Xi, d}}$. 
Let $E$ be a pure coherent sheaf, it is semistable if for every proper subsheaf
$F \subset E$ we have  $h_{\Xi}(F) \leq h_{\Xi}(E)$ and it is stable if the same is true with a strict inequality.
Then fixing a modified Hilbert polynomial $H$, the moduli stack of semistable coherent sheaves 
$\rM:=\rM^{\Xi}_H$ on $\sS$ is constructed in \cite{Nironi}.  If the stability and semistability coincide, the coarse moduli space 
$\rM$ is a projective scheme. 

The Higgs pair $(E,\phi)$  is semistable if for every proper $\phi$-invariant subsheaf
$F \subset E$ we have  $h_{\Xi}(F) \leq h_{\Xi}(E)$.  Let $\N:=\N_H$ be the moduli stack of stable Higgs pairs on $\sS$ with modified Hilbert polynomial $H$.  
Let $\sX:=\mbox{Tot}(K_{\sS})$ be the canonical line bundle of $\sS$, then $\sX$ is a smooth Calabi-Yau threefold DM stack.  
By spectral theory again, the category of Higgs pairs on $\sS$ is equivalent to the category of torsion sheaves $\sE_\phi$  on $\sX$ supporting on $\sS\subset \sX$.
Let $\pi: \sX\to \sS$ be the projection, then  the bullback $\pi^*\Xi$ is a generating sheaf for $\sX$.  One can take a projectivization $\overline{\sX}=\proj (K_{\sS}\oplus\sO_{\sS})$, and consider the moduli space of stable torsion sheaves 
on $\overline{\sX}$ with modified Hilbert polynomial $H$. The open part that is supported on the zero section  $\sS$ is isomorphic to the moduli stack of stable Higgs pairs $\N$ on $\sS$ with modified Hilbert polynomial $H$. 

There is a symmetric perfect obstruction theory on  the moduli space $\N^{\perp}_{L}$ of stable Higgs pairs 
$(E,\phi)$ with fixed determinant $L$ and trace-free on $\phi$, see \cite{JP}.    We define 
\begin{equation}\label{eqn_localized_SU_DM_invariants_intr}
\VW^L_{H}(\sS)=\int_{[(\N_L^{\perp})^{\cc^*}]^{\vir}}\frac{1}{e(N^{\vir})}.
\end{equation}
Also we have the Behrend function in this case 
and the invariant
$
\vw^L_{H}(\sS)=\chi(\N^{\perp}_{L}, \nu_{\N^{\perp}_{L}})
$
is the weighted Euler characteristic. 

\subsection{Calculations and checking S-duality}

The moduli space $\N_L^{\perp}$ admits a $\cc^*$-action induced by the $\cc^*$-action on the total space $\sX$ of the canonical line bundle $K_{\sS}$.  There are two type of $\cc^*$-fixed loci on $\N_L^{\perp}$ .  The first one corresponds to the $\cc^*$-fixed Higgs pairs $(E,\phi)$ such that the Higgs fields $\phi=0$.  Hence the fixed locus is just the moduli space $\rM_{L}(\sS)$ of stable torsion free sheaves $E$ on $\sS$. This is called the {\em Instanton Branch} as in \cite{TT1}.  The second type corresponds to $\cc^*$-fixed Higgs pairs $(E,\phi)$ such that the Higgs fields $\phi\neq 0$.  This case mostly happens when the surfaces $\sS$ are general type, and this component is called the {\em Monopole} branch.  See \S \ref{subsec_CStar_fixed_locus} for more details. 
In \cite{JP}, we calculate some invariants for the monopole branch of root stacks over a quintic surface $Q\subset \pp^3$; and quintic surface with ADE singularities;  and we survey some results here.  

We survey a check for the S-duality conjecture (\ref{eqn_S_transformation}) and  (\ref{eqn_S_transformation_2}) in rank two  for the projective plane $\pp^2$ based on the former calculation in \cite{GJK}, which is a result in a general proposal to attack the S-duality conjecture in \cite{Jiang_twist}.

\subsection{Outline} 
This survey paper is outlined as follows.  We review the Vafa-Witten invariants  for  a surface DM stack $\sS$  in \S \ref{sec_VW_surfaceDM}.
In \S \ref{sec_calculations} we survey the calculations on the surface DM stacks, where in \S \ref{subsec_root_stacks} we calculate the case of the $r$-th root stack $\sS$ over a smooth quintic surface $S$; and in \ref{subsec_quintic_ADE} we deal with the quintic surfaces with ADE singularities.  Finally in \S \ref{sec_S_duality_P2} we check the S-duality conjecture for the projective plane $\pp^2$ and discuss a proposal to attack the general surfaces. 

\subsection{Convention}
We work over $\cc$ throughout of the paper.   We use Roman letter $E$ to represent a coherent sheaf on a projective DM stack or a surface DM stack $\sS$, and use curl latter $\sE$ to represent the sheaves on the total space Tot$(\sL)$ of a line bundle $\sL$ over $\sS$. 
We reserve {\em $\rk$} for the rank of the torsion free coherent sheaves $E$, and use $\sqrt[r]{(S,C)}$ for the $r$-th root stack associated with the pair $(S, C)$ for a smooth projective surface and $C\subset S$ a smooth connected divisor. 

We keep the convention in \cite{VW} to use $SU(r)/\zz_r$ as the Langlands dual group of $SU(r)$.  When discussing the S-duality the rank $\rk=r$.

\subsection*{Acknowledgments}

Y. J. would like to thank the invitation by ICCM-2019, Tsinghua University Beijing,  for writing this survey paper, and thank Huai-Liang Chang, Huijun Fan,  Amin Gholampour, Shui Guo, Martijn Kool, Wei-Ping Li, Zhenbo Qin, Richard Thomas and Hsian-Hua Tseng for valuable discussions on the Vafa-Witten invariants and related fields.  This work is partially supported by  NSF DMS-1600997.


\section{Vafa-Witten theory for Deligne-Mumford surfaces}\label{sec_VW_surfaceDM}

In this section we work on a smooth surface DM stack $\sS$. The basic knowledge of stacks can be found in the book \cite{LMB}.   Several interesting examples and basic knowledge were reviewed in \cite{JP}.

\subsection{Moduli space of semistable Higgs  sheaves on surface DM stacks}\label{subsec_moduli_Higgs_pairs}

We choose the polarization $\sO_S(1)$ on the coarse moduli space $p: \sS\to S$. 

\begin{defn}\label{defn_very_ample}
A locally free sheaf $\Xi$ on $\sS$ is $p$-very ample if  for every geometric point of $\sS$ the representation of the stabilizer group at that point contains every irreducible representation of the stabilizer group. We call $\Xi$ a generating sheaf. 
\end{defn}
Let 
$\sO_S(1)$ be the very ample invertible sheaf on $S$, and $\Xi$ a generating sheaf on $\sS$. We call the pair $(\Xi, \sO_S(1))$ a polarization of $\sS$. 

Let us define the Gieseker stability condition:
\begin{defn}\label{defn_Gieseker}
The modified Hilbert polynomial of a coherent sheaf $F$ on $\sS$ is defined as:
$$H_{\Xi}(F,m)=\chi(\sS, F\otimes\Xi^{\vee}\otimes p^*\sO_{\sS}(m))
=\chi(S, F_{\Xi}(F)(m))$$
where $F_{\Xi}: D\Coh(\sS) \to D\Coh(\sS)$ is the functor defined by
$$F \mapsto p_*\cHom_{\sO_{\sS}}(\Xi, F).$$
\end{defn}
\begin{rmk}
\begin{enumerate}
\item  Let $F$ be of dimension $d$, then we can write:
$$H_{\Xi}(F,m)=\sum_{i=0}^{d}\alpha_{\Xi, i}(F)\frac{m^i}{i!}$$
which is induced by the case of schemes.
\item Also the modified Hilbert polynomial is additive on short exact sequences since the functor $F_{\Xi}$ is exact. 
\item If we don't choose the generating sheaf $\Xi$, the Hilbert polynomial $H$ on $\sS$ will be the same as the Hilbert polynomial on the coarse moduli space $S$.  In order to get interesting information on the DM stack $\sS$, the sheaf $\Xi$ is necessary.  For example, in \cite[\S 7]{Nironi}, \cite{Jiang2} the modified Hilbert polynomial on a root stack $\sS$ will corresponds to the parabolic Hilbert polynomial on the pair $(S,D)$ with $D\subset S$ a smooth divisor. 
\end{enumerate}
\end{rmk}

\begin{defn}\label{defn_reduced_hilbert}
The {\em reduced modified Hilbert polynomial} for the pure sheaf $F$ is defined as 
$$h_{\Xi}(F)=\frac{H_{\Xi}(F)}{\alpha_{\Xi, d}(F)}.$$
\end{defn}

\begin{defn}\label{defn_stability}
Let $F$ be a pure coherent sheaf.  We call $F$ semistable if for every proper subsheaf $F^\prime\subset F$, 
$$h_{\Xi}(F^\prime)\leq h_{\Xi}(F).$$
We call $F$ stable if $\leq$ is replaced by $<$ in the above inequality.
\end{defn}

\begin{defn}\label{defn_slope_stability}
 We define the slope of  $F$ by 
 $\mu_{\Xi}(F)=\frac{\alpha_{\Xi, d-1}(F)}{\alpha_{\Xi,d}(F)}$.
 Then $F$ is
semistable if for every proper subsheaf $F^\prime\subset F$, 
$\mu_{\Xi}(F^\prime)\leq \mu_{\Xi}(F)$.
We call $F$ stable if $\leq$ is replaced by $<$ in the above inequality.
\end{defn}

\begin{rmk}\label{rem_stability}
\begin{enumerate}
\item The notion of $\mu$-stability and semistability is related to the Gieseker stability and semistability in the same way as schemes, i.e.,
$$\mu-\text{stable}\Rightarrow \text{Gieseker stable}\Rightarrow \text{Gieseker semistable}\Rightarrow \mu-\text{semistable}$$
\item
The stability really depends on the generating sheaf $\Xi$.  This stability is not necessarily the same as the ordinary Gieseker stability even when $\sS$ is a scheme.  
\item 
One can define the rank 
$\rk F_{\Xi}(F)=\frac{\alpha_{\Xi, d}(F)}{\alpha_{d}(\sO_{S})}$.
\end{enumerate}
\end{rmk}

Let us fix a polarization $(\Xi, \sO_{S}(1))$ on $\sS$, and a modified Hilbert polynomial $H$.  There exists a moduli stack 
$\rM$ of semistable torsion free sheaves with Hilbert polynomial $H$ and $\rM$ is a global GIT quotient stack.  
 The coarse moduli space $\overline{\rM}$ of $\rM$ is 
 is a projective scheme.  Moreover, the stable locus $\overline{\rM}^{s}\subset \overline{\rM}$ is an open quasi-projective scheme.

The Higgs sheaves on $\sS$ is defined as follows. 
Let 
$\sX:=\mbox{Tot}(K_{\sS})$
be the total space of the canonical line bundle $K_{\sS}$ on $\sS$.  Since $\sS$ is a smooth two dimensional DM stack, $K_{\sS}$ exists as a line bundle.  The total space $\sX$ is a Calabi-Yau threefold DM stack.  

Let us fix a line bundle $\sL$ on $\sS$. 
A $\sL$-Higgs pair on $\sS$ is given by $(E,\phi)$, where $E\in \Coh(\sS)$ is a torsion free coherent sheaf and 
$$\phi\in \Hom(E, E\otimes \sL)$$
is a section.  We have:
\begin{prop}\label{prop_equivalent_categories}(\cite[Proposition 2.18]{JP})
There exists an abelian category $\Higg_{\sL}(\sS)$ of Higgs pairs on $\sS$ and an equivalence:
\begin{equation}\label{eqn_equivalence_categories}
\Higg_{\sL}(\sS)\stackrel{\sim}{\longrightarrow} \Coh_{c}(\sS)
\end{equation}
where $\Coh_{c}(\sS)$ is the category of compactly supported coherent sheaves on $\sX$. 
\end{prop}

The Gieseker stability on the Higgs pairs $(E,\phi)$ is similarly defined.  Let us fix a generating sheaf $\Xi$ on $\sS$. Then for any coherent sheaf $E\in\Coh(\sS)$ we have the modified Hilbert polynomial $h_{\Xi}(E)$.

\begin{defn}\label{defn_Gieseker_Higgs}
The $\sL$-Higgs pair $(E,\phi)$ is said to be Gieseker stable with respect to the polarization $(\Xi, \sO_{S}(1))$ if and only if 
$$h_{\Xi}(F)<h_{\Xi}(E)$$
for every proper $\phi$-invariant subsheaf $F\subset E$. 
\end{defn}

There exists a moduli stack 
$\N:=\N^H_{\Xi}(\sS)$ parametrizing stable Higgs sheaves
 with modified Hilbert polynomial $H$.
Then $\N$ is also represented by a GIT quotient stack with coarse moduli space a  quasi-projective scheme.

\subsection{Obstruction theory and the Vafa-Witten invariants}\label{sec_VW}

Let $\pi: \sX:=\mbox{Tot}(\sL)\to \sS$ be the projection from the total space of the line bundle $\sL$ to $\sS$. Then from the spectral theory a coherent sheaf $\sE$ on $\sX$ is equivalent to a 
$\pi_{*}\sO_{\sX}=\bigoplus_{i\geq 0}\mathcal{L}^{-i}\eta^i$-module, where $\eta$ is the tautological section of $\pi^*\mathcal{L}$. 

From \cite[\S 2.2]{TT1}, given a Higgs pair $(E, \phi)$, we have the torsion sheaf $\sE_{\phi}$ of $\sX$ supported on $\sS$.  $\sE_\phi$ is generated by its sections down on $\pi$ and we have a natural surjective morphism
\begin{equation}\label{eqn_rE_quotient}
0\to \pi^*(E\otimes\sL^{-1})\stackrel{\pi^*\phi-\eta}{\longrightarrow}\pi^*E=\pi^*\pi_*\sE_{\phi}\stackrel{\ev}{\longrightarrow}\sE_\phi\to 0
\end{equation}
with kernel $\pi^*(E\otimes\sL^{-1})$ as in Proposition 2.11 of \cite{TT1}.  All the arguments in 
\cite[Proposition 2.11]{TT1} work for smooth DM stack $\sS$ and $\sX$. 

The deformation of $\sE$ on $\sX$ is governed by $\Ext^*_{\sX}(\sE, \sE)$, while the Higgs pair $(E,\phi)$ is governed by the cohomology groups of the total complex
$$R\cHom_{\sS}(E, E)\stackrel{[\cdot, \phi]}{\longrightarrow}R\cHom_{\sS}(E,E\otimes\sL).$$
By some homological algebra proof as in \cite[Proposition 2.14]{TT1}, we have the exact triangle:
\begin{equation}\label{eqn_deformation1}
R\cHom(\sE_\phi,\sE_\phi)\to R\cHom_{\sS}(E, E)\stackrel{\circ\phi-\phi\circ}{\longrightarrow}R\cHom_{\sS}(E\otimes\sL^{-1},E).
\end{equation}
Taking cohomology of $(\ref{eqn_deformation1})$ we get 
\begin{equation}\label{eqn_coh_deformation1}
\cdots\to\Hom(E,E\otimes K_{\sS})\to \Ext^1(\sE_\phi, \sE_{\phi})\to \Ext^1(E,E)\to \cdots
\end{equation}
which relates the automorphisms, deformations and obstructions of $\sE_\phi$ to those of $(E,\phi)$.

Let $\sS\to B$ be a family of surface DM stacks $\sS$, i.e., a smooth projective morphism with the fibre surface DM stack, and let $\sX\to B$ be the total space of the a line bundle $\sL=K_{\sS/B}$. 
Let $\N^{H}$ denote the moduli space of Gieseker stable Higgs pairs on the fibre of $\sS\to B$ with fixed rank $r> 0$ and Hilbert polynomial $H$ (a fixed generating sheaf $\Xi$).    

We pick a (twisted by the $\cc^*$-action) universal sheaf $\rE$ over $\N\times_{B}\sX$.  We use the same $\pi$ to represent the projection 
$$\pi: \sX\to \sS; \quad  \pi: \N\times_{B}\sX\to \N\times_{B}\sS.$$
Since $\rE$  is flat over $\N$ and $\pi$ is affine, 
$$\E:=\pi_*\rE \text{~on~} \N\times_{B}\sS$$
is flat over $\N$.  $\E$ is also coherent because it can be seen locally on $\N$. Therefore it defines a classifying map:
$$\Pi: \N\to \rM$$
by
$$\sE\mapsto \pi_*\sE; \quad  (E,\phi)\mapsto E,$$
where $\rM$ is the moduli stack of coherent sheaves on the fibre of $\sS\to B$ with Hilbert polynomial $H$. 
For simplicity, we use the same $\E$ over $\rM\times \sS$ and $\E=\Pi^*\E$ on $\N\times \sS$. 
Let 
$$p_{\sX}:  \N\times_{B}\sX\to \N; \quad   p_{\sS}:  \N\times_{B}\sS\to \N$$
be the projections.  Then (\ref{eqn_deformation1}) becomes:
\begin{equation}\label{eqn_deformation2}
R\cHom_{p_{\sX}}(\rE, \rE)\stackrel{\pi_*}{\longrightarrow}R\cHom_{p_{\sS}}(\E, \E)\stackrel{[\cdot, \phi]}{\longrightarrow}R\cHom_{p_{\sS}}(\E, \E\otimes\sL).
\end{equation}
Let $\sL=K_{\sS/B}$ and
taking the relative Serre dual of the above exact triangle we get 
$$R\cHom_{p_{\sS}}(\E, \E)[2]\to R\cHom_{p_{\sS}}(\E, \E\otimes K_{\sS/B})[2]\to R\cHom_{p_{\sX}}(\rE, \rE)[3].$$
\begin{prop}(\cite[Proposition 2.21]{TT1})\label{prop_self_dual}
The above exact triangle  is the same as (\ref{eqn_deformation2}), just shifted.
\end{prop} 

Then the exact triangle  (\ref{eqn_deformation2}) fits into the following commutative diagram 
(\cite[Corollary 2.22]{TT1}):
\[
\xymatrix{
R\cHom_{p_{\sS}}(\E, \E\otimes K_{\sS/B})_{0}[-1]\ar[r]\ar@{<->}[d] &R\cHom_{p_{\sX}}(\rE, \rE)_{\perp}
\ar[r]\ar@{<->}[d] & R\cHom_{p_{\sS}}(\E, \E)_{0}\ar@{<->}[d]\\
R\cHom_{p_{\sS}}(\E, \E\otimes K_{\sS/B})[-1]\ar[r]\ar@{<->}[d]_{\id}^{\tr} &R\cHom_{p_{\sX}}(\rE, \rE)
\ar[r]\ar@{<->}[d] & R\cHom_{p_{\sS}}(\E, \E)\ar@{<->}[d]_{\id}^{\tr}\\
Rp_{\sS *}K_{\sS/B}[-1]\ar@{<->}[r]& Rp_{\sS *}K_{\sS/B}[-1]\oplus Rp_{\sS *}\sO_{\sS}\ar@{<->}[r]& Rp_{\sS *}\sO_{\sS}
}
\]
where $(-)_0$ denotes the trace-free Homs.  The $R\cHom_{p_{\sX}}(\rE, \rE)_{\perp}$ is the co-cone of the middle column and it will provide the symmetric obstruction theory of the moduli space $\N_{L}^{\perp}$ of stable trace free fixed determinant Higgs pairs.

From Proposition \ref{prop_self_dual}, in Appendix of  \cite{JP} we review that the truncation $\tau^{[-1,0]}R\cHom_{p_{\sX}}(\rE, \rE)$ defines a symmetric perfect obstruction theory on the moduli space $\N$. 
The total space $\sX=\mbox{Tot}(K_{\sS})\to \sS$ admits a $\cc^*$-action which has weight one on the fibres.  The obstruction theory is naturally $\cc^*$-equivariant. 
From \cite{GP}, the $\cc^*$-fixed locus $\N^{\cc^*}$ inherits a perfect obstruction theory
\begin{equation}\label{eqn_deformation_obstruction_fixed_locus}
\left(\tau^{[-1,0]}(R\cHom_{p_{\sX}}(\rE,\rE)[2])\Tt^{-1}\right)^{\cc^*}\to  \ll_{\N^{\cc^*}}
\end{equation}
by taking the fixed part.  Therefore it induces a virtual fundamental cycle 
$$[\N^{\cc^*}]^{\vir}\in H_*(\N^{\cc^*}).$$
The virtual normal bundle is given 
$$N^{\vir}:=\left(\tau^{[-1,0]}(R\cHom_{p_{\sX}}(\rE,\rE)[2]\Tt^{-1})^{\mov}\right)^{\vee}
=\tau^{[0,1]}(R\cHom_{p_{\sX}}(\rE,\rE)[1])^{\mov}$$
which is the derived dual of the moving part of $\tau^{[-1,0]}(R\cHom_{p_{\sX}}(\rE,\rE)[2])\Tt^{-1}$.

Consider the localized invariant
$$\int_{[\N^{\cc^*}]^{\vir}}\frac{1}{e(N^{\vir})}.$$
We explain this a bit.  Represent $N^{\vir}$ as a $2$-term complex $[E_0\to E_1]$ of locally free $\cc^*$-equivariant sheaves with non-zero weights and define 
$$e(N^{\vir}):=\frac{c_{\top}^{\cc^*}(E_0)}{c_{\top}^{\cc^*}(E_1)}\in H^*(\N^{\cc^*}, \zz)\otimes \qq[t, t^{-1}],$$
where $t=c_1(\Tt)$ is the generator of $H^*(B\cc^*)=\zz[t]$, and $c_{\top}^{\cc^*}$ denotes the $\cc^*$-equivariant top Chern class lying in $H^*(\N^{\cc^*}, \zz)\otimes_{\zz[t]} \qq[t, t^{-1}]$. 

\begin{defn}\label{defn_VW1}
Let $\sS$ be a smooth projective surface DM stack. Fixing a generating sheaf $\Xi$ on $\sS$, and a Hilbert polynomial $H$ associated with $\Xi$. Let $\N:=\N_H$ be the moduli space of stable Higgs pairs with  Hilbert polynomial $H$.  Then the primitive Vafa-Witten invariants of $\sS$ is defined as:
$$\widetilde{\VW}_{H}(\sS):=\int_{[\N^{\cc^*}]^{\vir}}\frac{1}{e(N^{\vir})}\in \qq$$
which is referred as $U(\rk)$-Vafa-Witten invariants. 
\end{defn}

\begin{rmk}
We have 
$$\Ext^\bullet_{\sX}(\sE_\phi, \sE_\phi)=H^{\bullet-1}(K_{\sS})\oplus H^{\bullet}(\sO_{\sS})\oplus 
\Ext^{\bullet}_{\sX}(\sE_\phi, \sE_\phi)_{\perp},$$
where $\Ext^{\bullet}_{\sX}(\sE_\phi, \sE_\phi)_{\perp}$ is the trace zero part with determinant $L\in \Pic(\sS)$. 
Hence the obstruction sheaf in the obstruction theory (\ref{eqn_deformation_obstruction_fixed_locus}) has a trivial summand $H^2(\sO_{\sS})$. 
So $[\N^{\cc^*}]^{\vir}=0$ is $h^{0,2}(\sS)>0$.  If $h^{0,1}(\sS)\neq 0$, then tensoring with flat line bundle makes the obstruction theory invariant. Therefore the integrand is the  pullback from $\N/\Jac(\sS)$, which is a lower dimensional space, hence zero. 
\end{rmk}

\subsection{$SU(\rk)$ Vafa-Witten invariants}

Let us now fix $(L, 0)\in \Pic(\sS)\times \Gamma(K_{\sS})$, and let $\N^{\perp}_{L}$ be the fibre of 
$$\N/\Pic(\sS)\times\Gamma(K_{\sS}).$$
Then moduli space $\N^{\perp}_{L}$ of stable Higgs sheaves $(E,\phi)$ with $\det(E)=L$ and trace-free $\phi\in \Hom(E,E\otimes K_{\sS})_0$ admits a symmetric obstruction theory 
$$R\cHom_{p_{\sX}}(\rE, \rE)_{\perp}[1]\Tt^{-1}\longrightarrow \ll_{\N^{\perp}_{L}}.$$

\begin{defn}\label{defn_SU_VW_invariants}
Let $\sS$ be a smooth projective surface DM stack.  Fix a generating sheaf $\Xi$ for $\sS$, and a Hilbert polynomial $H$ associated with $\Xi$.  Let $\N^{\perp}_{L}:=\N_{L}^{\perp, H}$ be the moduli space of stable Higgs sheaves with Hilbert polynomial $H$.  Then define
$$\VW_{H}(\sS):=\int_{[(\N^{\perp}_{L})^{\cc^*}]^{\vir}}\frac{1}{e(N^{\vir})}.$$
\end{defn}

Since we work on surface DM stack $\sS$, it maybe better to fix the K-group class $\mathbf{c}\in K_0(\sS)$ such that the Hilbert polynomial of $\mathbf{c}$ is $H$.  Then 
$\VW_{\mathbf{c}}(\sS)=\int_{[(\N^{\perp}_{L})^{\cc^*}]^{\vir}}\frac{1}{e(N^{\vir})}$ is Vafa-Witten invariant corresponding to $\mathbf{c}$.

\subsection{$\cc^*$-fixed loci}\label{subsec_CStar_fixed_locus}

We discuss the $\cc^*$-fixed loci for the moduli space $\N^{\perp}_{L}$.

\subsubsection{Case I-Instanton Branch:}\label{subsubsec_first_type}
For the Higgs pairs $(E,\phi)$ such that $\phi=0$, the $\cc^*$-fixed locus is exactly the moduli space $\rM_{L}$ of Gieseker stable sheaves on $\sS$ with fixed determinant $L$ and with Hilbert polynomial $H$ associated with the generating sheaf $\Xi$.  The exact triangle in (\ref{eqn_deformation2}) splits the obstruction theory
$$R\cHom_{p_{\sX}}(\rE, \rE)_{\perp}[1]\Tt^{-1}\cong R\cHom_{p_{\sS}}(\E, \E\otimes K_{\sS})_{0}[1]\oplus 
R\cHom_{p_{\sS}}(\E, \E)_{0}[2]\Tt^{-1}$$
where $\Tt^{-1}$ represents the moving part of the $\cc^*$-action. Then the $\cc^*$-action induces a perfect obstruction theory 
$$E_{\rM}^{\bullet}:=R\cHom_{p_{\sS}}(\E, \E\otimes K_{\sS})_{0}[1]\to \ll_{\rM_{L}}.$$
The virtual normal bundle 
$$N^{\vir}=R\cHom_{p_{\sS}}(\E, \E\otimes K_{\sS})_{0}\Tt=E_{\rM}^{\bullet}\otimes \Tt[-1].$$
So the invariant contributed from $\rM_{L}$ (we can let $E_{\rM}^{\bullet}$ is quasi-isomorphic to $E^{-1}\to E^0$) is: 
\begin{align*}
\int_{[\rM_{L}]^{\vir}}\frac{1}{e(N^{\vir})}&=\int_{[\rM_{L}]^{\vir}}\frac{c_s^{\cc^*}(E^0\otimes \Tt)}{c_r^{\cc^*}(E^{-1}\otimes \Tt)}\\
&=\int_{[\rM_{L}]^{\vir}}\frac{c_s(E^0)+\Tt c_{s-1}(E^0)+\cdots}{c_r(E^{-1})+\Tt c_{r-1}(E^{-1})+\cdots}
\end{align*}
Here we assume $r$ and $s$ are the ranks of $E^{-1}$ and $E^0$ respectively, and $r-s$ is the virtual dimension of 
$\rM_{L}:=\rM_{L,H}$.  By the virtual dimension consideration, only $\Tt^0$ coefficient contributes and we may let $\Tt=1$, so
\begin{align}\label{eqn_virtual_Euler_number}
\int_{[\rM_{L}]^{\vir}}\frac{1}{e(N^{\vir})}&=\int_{[\rM_{L}]^{\vir}}\Big[\frac{c_{\bullet}(E^0)}{c_{\bullet}(E^{-1})}\Big]_{\vd}\nonumber \\
&=\int_{[\rM_{L}]^{\vir}}c_{\vd}(E_{\rM}^{\bullet})\in \zz.
\end{align}
This is the signed virtual Euler number of Ciocan-Fontanine-Kapranov/Fantechi-G\"ottsche. 
We have the following result: 
\begin{prop}\label{prop_K_S_fixed_locus}(\cite[Proposition 3.6]{JP})
Let us fix a generating sheaf $\Xi$ on $\sS$.  If $\deg K_{\sS}\leq 0$, then any stable $\cc^*$-fixed Higgs pair 
$(E,\phi)$ has Higgs field $\phi=0$. Therefore if we fix some $K$-group class $\mathbf{c}\in K_0(\sS)$, then 
$\VW^{L}_{\mathbf{c}}(\sS)$ is the same as the signed virtual Euler number in (\ref{eqn_virtual_Euler_number}).
\end{prop}

Also we have:

\begin{prop}
If $\deg K_{\sS}< 0$, then any semistable $\cc^*$-fixed Higgs pair $(E,\phi)$ has Higgs field $\phi=0$. 
\end{prop}
\begin{proof}
This is the same as Proposition \ref{prop_K_S_fixed_locus}.
\end{proof}

\subsubsection{Case II-Monopole Branch:}\label{subsec_second_fixed_loci}
The second component $\rM^{(2)}$ corresponds to the Higgs fields $\phi\neq 0$. Let $(E,\phi)$ be a $\cc^*$-fixed stable Higgs pair.  Since the $\cc^*$-fixed stable sheaves $\sE_{\phi}$ are simple,  we  use \cite[Proposition 4.4]{Kool}, \cite{GJK} to make this stable sheaf 
$\cc^*$-equivariant.  The cocycle condition in the  $\cc^*$-equivariant definition for the Higgs pair $(E,\phi)$ corresponds to a $\cc^*$-action 
$$\psi: \cc^*\to \Aut(E)$$
such that 
\begin{equation}\label{eqn_cocycle_condition}
\psi_{t}\circ \phi\circ \psi_t^{-1}=t\phi
\end{equation}  
With respect to the $\cc^*$-action on $E$,  it splits into a direct sum of eigenvalue subsheaves
$$E=\oplus_{i}E_i$$
where $E_i$ is the weight space such that $t$ has by $t^i$, i.e., $\psi_t=\mbox{diag}(t^i)$. 
The action acts on the Higgs field with weight one by (\ref{eqn_cocycle_condition}).
Also for a Higgs pair $(E,\phi)$, if a $\cc^*$-action on $E$ induces weight one action on $\phi$, then it is a fixed point of the $\cc^*$-action.

Since the $\cc^*$-action on the canonical line bundle $K_{\sS}$ has weight $-1$, $\phi$ decreases the weights, and it maps the lowest weight torsion subsheaf to zero, hence zero by stability. 
So each $E_i$ is torsion free and have rank $> 0$.  Thus $\phi$ acts blockwise through morphisms
$$\phi_i: E_i\to E_{i-1}.$$
These are flags of torsion-free sheaves on $\sS$, see \cite{TT1}. 

In the case that $E_i$ has rank $1$, they are twisted by line bundles, and $\phi_i$ defining nesting of ideals. Then this is the nested Hilbert scheme on $\sS$. 
Also it is interesting to see when there exist rank $1$ torsion free sheaves on a surface DM stacks. 
Very little is known of nested Hilbert schemes for surface DM stacks.

\section{Calculation results}\label{sec_calculations}

We survey some  calculation results  on two type of  general type surface DM stacks, one is for a $r$-root stack  over a smooth quintic surface, and the other is for quintic surface with ADE singularities.  

\subsection{Root stack on quintic surfaces}\label{subsec_root_stacks}

Let $S\subset \pp^3$ be a smooth quintic surface in $\pp^3=\proj(\cc[x_0:x_1:x_2,x_3])$, given by a homogeneous degree $5$ polynomial.  Let $C\subseteq |K_{S}|$ be a smooth connected canonical divisor such that 
$\sO_S$ is the only line bundle $L$ satisfying $0\leq \deg L\leq \frac{1}{2}\deg K_{S}$ where the degree is defined  by 
$\deg L=c_1(L)\cdot c_1(\sO_S(1))$.  Then we have the following topological invariants:
\begin{equation}\label{eqn_topology_invariants_quintic}
\begin{cases}
g_{C}=1+c_1(S)^2=1+5=6;\\
h^0(K_{S})=p_{g}(S)=\frac{1}{12}(c_1(S)^2+c_2(S))-1=\frac{1}{12}(5+55)-1=4;\\
h^0(K_S^2)=p_g(S)+g_C=10.
\end{cases}
\end{equation}
Let 
$\sS:=\sqrt[r]{(S,C)}$
be the root stack associated with the divisor $C$.  One can take $\sS=\sqrt[r]{(S,C)}$ as the $r$-th root stack associated with the line bundle $\sO_S(C)$.  Let 
$p: \sS\to S$
be the projection to its coarse moduli space $S$, and let 
$$\sC:=p^{-1}(C).$$
We still use $p: \sC\to C$ to represent the projection and it is a $\mu_r$-gerbe over $C$. 
The canonical line bundle $K_{\sS}$ satisfies the formula 
$$K_{\sS}=p^*K_{S}+\frac{r-1}{r}\sO_{\sS}(\sC)=\sO_{\sS}(\sC).$$

Recall that $\sX=\mbox{Tot}(K_{\sS})$, and $X:=\mbox{Tot}(K_{S})$, and let 
$$\pi: \sX\to \sS; \quad  \pi: X\to S$$
be the projection. 
We pick the generating sheaf ``$\Xi=\oplus_{i=0}^{r}\sO_{\sS}(i\sC^{\frac{1}{r}})$", and a Hilbert polynomial $H$, and let 
$\N_H$ be the moduli space of stable Higgs sheaves on $\sS$ with Hilbert polynomial $H$. 

The $\cc^*$ acts on $\sX$ by scaling the fibres of $\sX\to \sS$.  Let $(E,\phi)$ be a $\cc^*$-fixed rank $2$ Higgs pair with fixed determinant $L=K_{\sS}$ in the second component $\rM^{(2)}$ in \S \ref{subsec_second_fixed_loci}.   
Then since all the $E_i$ have rank bigger than zero, 
$$E=E_i\oplus E_j.$$
Without loss of generality, we may let 
$E=E_0\oplus E_{-1}$ since tensoring $E$ by $\Tt^{-i}$ $E_i$ goes to $E_0$, where $\Tt$ is the standard one dimensional $\cc^*$-representation of weight one. 
Then considering $\phi$ as a weight zero element of $\Hom(E,E\otimes K_{\sS})\otimes \Tt$, we have 
$$E=E_0\oplus E_{-1}, \text{~and~} \phi= \left(
\begin{array}{cc}
0&0\\
\iota&0
\end{array}
\right)$$
for some 
$\iota: E_0\to E_{-1}\otimes K_{\sS}\otimes \Tt$. Then $E_{-1}\hookrightarrow E$ is a $\phi$-invariant subsheaf, and by semistability (Gieseker stable implies $\mu$-semistable)
we have 
$$\mu_{\Xi}(E_{-1})\leq \mu_{\Xi}(E_0)=\mu_{\Xi}(K_{\sS})-\mu_{\Xi}(E_{-1}).$$
The existence of the nonzero map $\Phi: E_0\to E_{-1}\otimes K_{\sS}$ implies:
$$\mu_{\Xi}(E_{-1})+\mu_{\Xi}(K_{\sS})\geq \mu_{\Xi}(E_0)=\mu_{\Xi}(K_{\sS})-\mu_{\Xi}(E_{-1}).$$
So 
\begin{equation}\label{eqn_inequality}
0\leq \mu_{\Xi}(E_{-1})\leq \frac{1}{2}\mu_{\Xi}(K_{\sS}).
\end{equation}

\begin{lem}\label{lem_E0E1}(\cite[Lemma 4.2]{JP})
The inequality (\ref{eqn_inequality}) implies that 
$$\det(E_{-1})=\sO_{\sS}; \text{~and~} \det(E_{0})=K_{\sS}.$$
\end{lem}

The lemma implies that 
$$E_0=\sI_0\otimes K_{\sS}, \quad  E_{-1}=\sI_1\otimes \Tt^{-1}$$
for some ideal sheaves $\sI_i$.  The morphism $\sI_0\to \sI_1$ is nonzero, so we must have:
$\sI_0\subseteq \sI_1$.
So there exist $\sZ_1\subseteq \sZ_0$ two zero-dimensional subsheaves parametrized by $\sI_0\subseteq \sI_1$. \\

\textbf{Components in terms of $K$-group class}\label{subsec_components_K-group}

Let $K_0(\sS)$ be the Grothendieck $K$-group of $\sS$, and we want to use Hilbert scheme on $\sS$ parametrized by 
$K$-group classes.  We fix the filtration
$$F_0K_0(\sS)\subset F_1K_0(\sS)\subset F_2K_0(\sS)$$
where $F_iK_0(\sS)$ is the subgroup of $K_0(\sS)$ such that the support of the elements in $F_iK_0(\sS)$ has dimension $\leq i$. 
The orbifold Chern character morphism is defined by:
\begin{equation}\label{eqn_orbifold_Chern_character}
\widetilde{\Ch}: K_0(\sS)\to H^*_{\CR}(\sS,\qq)=H^*(I\sS, \qq)
\end{equation}
where $H^*_{\CR}(\sS,\qq)$ is the Chen-Ruan cohomology of $\sS$.  The inertia stack 
$$I\sS=\sS\bigsqcup \sqcup_{i=1}^{r-1}\sC_i$$
where each $\sC_i=\sC$ is the stacky divisor of $\sS$.   We should understand that the inertia stack 
is indexed by the element $g\in\mu_r$, $\sS_{g}\cong \sC$ is the component corresponding to $g$.  It is clear that 
$\sS_{1}=\sS$ and $\sS_{g}=\sC$ if $g\neq 1$. 
Let $\zeta\in \mu_r$ be the generator of $\mu_r$. 
Then 
$$H^*(I\sS, \qq)=H^*(\sS)\bigoplus \oplus_{i=1}^{r-1}H^*(\sC_i),$$
where $\sC_i$ corresponds to the element $\zeta^i$. 
The cohomology of $H^*(\sC_i)$ is isomorphic to $H^*(C)$. 
For any coherent sheaf $E$, the restriction of $E$ to every $\sC_i$ has a $\mu_r$-action. 
We assume that 
$$E=E_{\zeta^i}^1\oplus E_{\zeta^i}^2$$
is the decomposition of eigen-subsheaves such that it acts by 
$e^{2\pi i\frac{f_{i1}}{r}}$ on $E_{\zeta^i}^1$ and $e^{2\pi i\frac{f_{i2}}{r}}$ on $E_{\zeta^i}^2$. 
We  let 
\begin{equation}\label{eqn_Chern_character_value1}
\widetilde{\Ch}(E)=(\Ch(E), \oplus_{i=1}^{r-1}\Ch(E|_{\sC_i})),
\end{equation}
where 
$$\Ch(E)=(\rk(E), c_1(E), c_2(E))\in H^*(\sS),$$
and 
$$
\Ch(E|_{\sC_i})=\left(e^{2\pi i\frac{f_{i1}}{r}}+e^{2\pi i\frac{f_{i2}}{r}}, e^{2\pi i\frac{f_{i1}}{r}}c_1(E_{\zeta^i}^1)
+e^{2\pi i\frac{f_{i2}}{r}}c_1(E_{\zeta^i}^2)\right)\in H^*(\sC_i).$$

In order to write down the generating function later.  We introduce some notations.  We roughly write 
$$\widetilde{\Ch}(E)=(\widetilde{\Ch}_{g}(E))$$
where $\widetilde{\Ch}_{g}(E)$ is the component in $H^*(\sS_{g})$ as in (\ref{eqn_Chern_character_value1}). 
Then define:
\begin{equation}\label{eqn_Chern_character_degree}
\left(\widetilde{\Ch}_{g}\right)^k:=\left(\widetilde{\Ch}_{g}\right)_{\dim \sS_{g}-k}\in H^{\dim \sS_{g}-k}(\sS_{g}).
\end{equation}
The $k$ is called the codegree in \cite{GJK}.  In our inertia stack $\sS_g$ is either the whole $\sS$, or $\sC$,
therefore if we have a rank $2$ $\cc^*$-fixed Higgs pair $(E,\phi)$ with fixed $c_1(E)=-c_1(\sS)$, then 
$\left(\widetilde{\Ch}_{g}\right)^2(E)=2$, the rank; while 
$$
\left(\widetilde{\Ch}_{g}\right)^1(E)=
\begin{cases}
-c_1(\sS), & g=1;\\
 e^{2\pi i\frac{f_{i1}}{r}}+ e^{2\pi i\frac{f_{i2}}{r}}, & g=\zeta^i \neq 1.
\end{cases}
$$
Also we have 
$$
\left(\widetilde{\Ch}_{g}\right)^0(E)=
\begin{cases}
c_2(E), & g=1;\\
e^{2\pi i\frac{f_{i1}}{r}}c_1(E_{\zeta^i}^1)
+e^{2\pi i\frac{f_{i2}}{r}}c_1(E_{\zeta^i}^2), & g=\zeta^i \neq 1.
\end{cases}
$$

Therefore we have the following proposition:
\begin{prop}\label{prop_second_fixed_loci_Hilbert_scheme}(\cite[Proposition 4.3]{JP})
In the case that the rank of stable Higgs sheaves is $2$, we fix a $K$-group class 
$\mathbf{c}\in K_0(\sS)$ such that $\left(\widetilde{\Ch}_{1}\right)^1(\mathbf{c})=-c_1(\sS)$.  Then 
\begin{enumerate}
\item If $c_2(E)<0$, then the $\cc^*$-fixed locus is empty by the assumption of Bogomolov inequality.  
\item If $c_2(E)\geq 0$,  then 
$$\rM^{(2)}\cong \bigsqcup_{\alpha\in F_0K_0(\sS)}\Hilb^{\alpha, \mathbf{c}_0-\alpha}(\sS)$$
where $\mathbf{c_0}\in F_0K_0(\sS)$ such that $\left(\widetilde{\Ch}_{g}\right)^0(\mathbf{c}_0)=\left(\widetilde{\Ch}_{g}\right)^0(\mathbf{c})$; and $\Hilb^{\alpha, \mathbf{c}_0-\alpha}(\sS)$ is the nested Hilbert scheme of zero-dimensional substacks of $\sS$:
$$\sZ_1\subseteq \sZ_0$$
such that $[\sZ_1]=\alpha$, $[\sZ_0]=\mathbf{c}_0-\alpha$.  $\square$
\end{enumerate}
\end{prop}

Several cases are calculated.
\subsubsection{The case $\sZ_1=\emptyset$}\label{subsec_case_Z1_empty}

Therefore in this case 
$$E=\sI_0\otimes K_{\sS}\oplus \sO\cdot \Tt^{-1}.$$
So the nested Hilbert scheme $\Hilb^{\alpha, \mathbf{c}_0-\alpha}(\sS)$ is just the Hilbert scheme 
$\Hilb^{\mathbf{c}_0}(\sS)$ on $\sS$.   One can study the  deformation and obstruction theory in detail as in \cite[\S 4.1.3]{JP} to calculate that the virtual fundamental class is given by:
\begin{equation}\label{eqn_Euler_obstruction_bundle2}
[\rM^{(2)}]^{\vir}
=(-1)^{\rk}\sC^{[\mathbf{c}_0]}\subset  \Hilb^{\mathbf{c}_0}(\sS)=\rM^{(2)}.
\end{equation}

The calculation of the virtual normal bundle $N^{\vir}$  of $\rM^{(2)}$ is the same as in \cite[\S 8.3]{TT1}, which is given by the moving part of the obstruction theory:
$$\Gamma(K_{\sS}|_{\sZ_0})\Tt\oplus R\Gamma(\sI_0 K_{\sS}^2)\Tt^2\oplus R\Gamma(\sI_0 K^2_{\sS})^{\vee}\Tt^{-1}[-1]\oplus T_{\sZ_0}^*\Hilb^{\mathbf{c}_0}(\sS)\Tt[-1]$$
at $\sZ_0\in \rM^{(2)}$.  Then we calculate the virtual normal bundle $N^{\vir}$ by noting that
$$R\Gamma(\sI_0 K_{\sS}^2)=H^0(K_{\sS^2})-H^0(K_{\sS}|_{\sZ_0});$$
and $N^{\vir}$ is:
$$[K_{\sS}^{[\mathbf{c}_0]}]\Tt+(\Tt^2)^{\oplus\dim H^0(K_{\sS}^2)}
-[(K_{\sS}^2)^{[\mathbf{c}_0]}]\Tt^2-(\Tt^{-1})^{\oplus\dim H^0(K_{\sS}^2)}
+[((K_{\sS}^2)^{[\mathbf{c}_0]})^*]\Tt^{-1}-\Big[T_{\Hilb^{\mathbf{c}_0}(\sS)}\Big]\Tt.$$

Since $\sC\to C$ is a $\mu_r$-gerbe, we can just write the Hilbert scheme $\sC^{[\mathbf{c}_0]}$ as 
$\sC^{[n]}$ for some integer $n\in \zz_{\geq 0}$. 
So we calculate the virtual Euler class:
\begin{align*}
\frac{1}{e(N^{\vir})}&=\frac{e((K_{\sS})^{[n]})\Tt^2)\cdot e(\Tt^{-1})^{\oplus\dim H^0(K_{\sS}^2)}\cdot e(T^*_{\Hilb^{\mathbf{c}_0}(\sS)}\Tt)}
{e(K_{\sS}^{[n]}\Tt)\cdot e((\Tt^2)^{\oplus \dim H^0(K_{\sS}^2)})\cdot e(((K_{\sS}^2)^{[n]})^* \Tt^{-1})}\\
&=\frac{(2t)^n\cdot c_{\frac{1}{2t}}((K_{\sS}^2)^{[n]})\cdot (-t)^{\dim H^0(K_{\sS}^2)}\cdot t^{2n}\cdot c_{\frac{1}{t}}(T^*_{\Hilb^n(\sS)})}
{t^n\cdot c_{\frac{1}{t}}((K_{\sS})^{[n]})\cdot (2t)^{\dim H^0(K_{\sS}^2)}\cdot (-1)^n\cdot t^n c_{\frac{1}{t}}((K_{\sS}^2)^{[n]})}\\
&=(-2)^{n-\dim}\cdot t^n\cdot 
\frac{c_{\frac{1}{2t}}((K_{\sS}^2)^{[n]})\cdot c_{-\frac{1}{t}}(T_{\Hilb^n(\sS)})}
{c_{\frac{1}{t}}((K_{\sS})^{[n]})\cdot c_{\frac{1}{t}}((K_{\sS}^2)^{[n]})}
\end{align*}
where 
$$c_s(E):=1+sc_1(E)+\cdots+s^r c_r(E),$$
and when $s=1$, $c_s(E)$ is the total Chern class of $E$.  By the arguments of the degree, we calculate the case $t=1$. 
Also since $\sC^{[n]}$ is cut out of the section $s^{[n]}$ on $K_{\sS}^{[n]}$, we have 
$$T_{\Hilb^{\mathbf{c}_0}}(\sS)|_{\sC^{[n]}}=T_{\sC^{[n]}}\oplus K_{\sS}^{[n]}|_{\sC^{[n]}}$$
in $K$-theory. 
Therefore
\begin{equation}\label{eqn_rM_integral}
\int_{[\rM^{(2)}]^{\vir}}\frac{1}{e(N^{\vir})}=(-2)^{-\dim}\cdot 2^n\cdot 
\int_{[\sC^{[n]}]}\frac{c_{\frac{1}{2}}((K_{\sS}^2)^{[n]})\cdot c_{-1}(T_{\sC^{[n]}})\cdot c_{-1}(K_{\sS}^{[n]})}
{c_{\bullet}(K_{\sS}^{[n]})\cdot c_{\bullet}((K_{\sS}^2)^{[n]})}.
\end{equation}

For the  $\mu_r$-gerby curve $\sC\to C$, we have

\begin{prop}\label{prop_Hilbert_scheme_gerbe}(\cite[Proposition 4.6]{JP})
Let $\Hilb^n(\sC)_{\rho}$ be the Hilbert scheme of points on $\sC$ parametrizing the representation $\rho$ of $\mu_r$ of length $n$.  Then $\sC^{[n]}_{\rho}:=\Hilb^n(\sC)_{\rho}$ is a $(\mu_r)^n$-gerbe over the Hilbert scheme of $n$-points 
$C^{[n]}$ on $C$, and we denote by 
$$p^{[n]}: \sC^{[n]}_{\rho}\to C^{[n]}$$
the structure morphism. 
\end{prop}
We also have the following results of the line bundles on the Hilbert scheme of points on the gerby curve. 
\begin{prop}\label{prop_pullback_linebundle_Hilbert_scheme}(\cite[Proposition 4.7]{JP})
Let $\sL\to C$ be a line bundle over $C$ and $\sL^{[n]}\to C^{[n]}$ the vector bundle induced by $\sL$. 
For the structure morphism  
$$p^{[n]}: \sC^{[n]}\to C^{[n]},$$ 
we have
$$(p^*\sL)^{[n]}\cong (p^{[n]})^*(\sL^{[n]}).$$
\end{prop}

From Proposition \ref{prop_Hilbert_scheme_gerbe} and Proposition \ref{prop_pullback_linebundle_Hilbert_scheme}, one can use the calculation for the Hilbert scheme of points on $C$ in \cite[\S 8.4]{TT1} to calculate the integral on the Hilbert scheme of points on the gerby curve $\sC$.
Let us first review the tautological classes on $C^{[n]}$ for the smooth curve $C$. 
Let 
$$\omega:=\PD[C^{[n-1]}]\in H^2(C^{[n]},\zz)$$
where $C^{[n-1]}\subset C^{[n]}$ is a smooth divisor given by $Z\mapsto Z+x$ for a base point $x\in C$, and $\PD$ represents the Poincare dual. 
The second one is given by the Abel-Jacobi map:
$$\AJ: C^{[n]}\to \Pic^n(C); \quad   Z\mapsto \sO(Z).$$
Since tensoring with power of $\sO(x)$ makes the $\Pic^n(C)$ isomorphic for all $n$, the pullback of the theta divisor from $\Pic^{g-1}(C)$  gives a cohomology class 
$$\theta\in H^2(\Pic^n(C),\zz)\cong \Hom(\Lambda^2H^1(C,\zz),\zz).$$
Still let $\theta$ to denote its pullback 
$\AJ^*\theta$, so 
$$\theta\in H^2(C^{[n]},\zz),$$
which is the second tautological class.  The basic property (\cite[\S I.5]{ACGH}) is:
\begin{equation}\label{eqn_basic_formula1}
\int_{C^{[n]}}\frac{\theta^i}{i!}\omega^{n-i}=\mat{c}g\\
i\rix,
\end{equation}
and 
$$
\begin{cases}
c_t(T_{C^{[n]}})=(1+\omega t)^{n+1-g}\exp\left(\frac{-t\theta}{1+\omega t}\right);\\
c_t(\sL^{[n]})=(1-\omega t)^{n+g-1-\deg \sL}\exp\left(\frac{t\theta}{1-\omega t}\right).
\end{cases}
$$
So from Proposition \ref{prop_Hilbert_scheme_gerbe}, 
$$(p^{[n]})^*\omega=r^{n}\omega; \quad   (p^{[n]})^*\theta=r^{n}\theta.$$

\begin{lem}\label{lem_two_formula_for_gerbe_Hilbert_scheme}(\cite{JP})
We have 
$$
\begin{cases}
c_t(T_{\sC^{[n]}})=(1+r^{n}\omega t)^{n+1-g}\exp\left(\frac{-t\cdot r^n\cdot\theta}{1+r^n\omega t}\right);\\
c_t((p^*\sL)^{[n]})=(1-r^n\omega t)^{n+g-1-\deg \sL}\exp\left(\frac{t r^n\theta}{1-r^n\omega t}\right).
\end{cases}
$$
\end{lem}

Thus we calculate
\begin{align}\label{eqn_key_formula1}
&\text{Right side of ~} (\ref{eqn_rM_integral})=   \\
&(-2)^{\dim}\cdot 2^n\int_{\sC^{[n]}}
\frac{(1-\frac{r^n\omega}{2})^{n+1-g}\cdot e^{\frac{r^n\theta}{2-r^n\omega}}\cdot (1-r^n\omega)^{n+1-g}\cdot e^{\frac{r^n\theta}{1-r^n\omega}}\cdot (1+r^n\omega)^n\cdot e^{\frac{-r^n\theta}{1+r^n\omega}}}
{(1-r^n\omega)^n\cdot e^{\frac{r^n\theta}{1-r^n\omega}}\cdot (1-r^n\omega)^{n+1-g}\cdot e^{\frac{r^n\theta}{1-r^n\omega}}}\nonumber \\
&=(-2)^{\dim}\cdot 2^{g-1}(-1)^{n+1-g}\cdot \int_{\sC^{[n]}}
(r^n\omega-2)^{n+1-g}\cdot \frac{(1+r^n\omega)^n}{(1-r^n\omega)^n}\cdot 
e^{\frac{r^n\theta}{2-r^n\omega}-\frac{r^n\theta}{1+r^n\omega}-\frac{r^n\theta}{1-r^n\omega}}. \nonumber
\end{align}

Now we use (\ref{eqn_basic_formula1}) and Proposition \ref{prop_Hilbert_scheme_gerbe} to get 
\begin{align*}
\int_{\sC^{[n]}}\frac{(r^n\theta)^i}{i!}\cdot (r^n\omega)^{n-i}
=\int_{C^{[n]}} \frac{(\theta)^i}{i!}\cdot (\omega)^{n-i}=\mat{c}g\\i\rix.
\end{align*}
Hence whenever we have $\frac{(r^n\theta)^i}{i!}$ in the integrand involving only power of $\omega$ we can replace it by 
$\mat{c}g\\i\rix (r^n\omega)^{n-i}$.  Therefore for $\alpha$ a power series of $\omega$, 
$$e^{\alpha (r^n \theta)}=\sum_{i=0}^{\infty}\alpha^i \frac{(r^n\theta)^i}{i!}\sim
\sum_{i=0}^{\infty}\alpha^i \mat{c}g\\i\rix (r^n\omega)^{n-i}=(1+\alpha (r^n\omega))^g.$$

When we do the integration against $\sC^{[n]}$, $\sim$ becomes equality, and 
(\ref{eqn_key_formula1}) is: 
\begin{align}\label{eqn_key_formula2}
&(-2)^{\dim}\cdot 2^{g-1}(-1)^{n+1-g}\cdot \int_{\sC^{[n]}}
(r^n\omega-2)^{n+1-g}\cdot \frac{(1+r^n\omega)^n}{(1-r^n\omega)^n}\cdot 
(1+\frac{r^n\omega}{2-r^n\omega}-\frac{r^n\omega}{1+r^n\omega}-\frac{r^n\omega}{1-r^n\omega})^g
 \\
&=(-2)^{\dim}\cdot (-2)^{g-1}(-1)^{n}\cdot \int_{\sC^{[n]}}
(r^n\omega-2)^{n+1-2g}\cdot \frac{(1+r^n\omega)^{n-g}}{(1-r^n\omega)^{n+g}}\cdot (4r^n\omega-2)^g. \nonumber
\end{align}

\subsubsection{Writing the generating function}\label{subsec_generating_function}

Recall that for the $\cc^*$-fixed Higgs pair $(E,\phi)$, we fix 
$$(\widetilde{\Ch}_1)^2(E)=2. \quad (\widetilde{\Ch}_1)^1(E)=-c_1(\sS).$$
For $g_i=\zeta^i\in\mu_r (1\leq i\leq r-1)$, $E=E_{g_i}^1\oplus E_{g_i}^2$ is the decomposition under the $g_i$-action into eigen-subsheaves.  We calculate and denote by 
\begin{equation}\label{eqn_Ci_C1}
\beta_{g_i}:=(\widetilde{\Ch}_{g_i})^1(E)=e^{2\pi i \frac{f_{i1}}{r}}+e^{2\pi i \frac{f_{i2}}{r}}.
\end{equation}
Then we let
\begin{equation}\label{eqn_Ci_C2}
n_{i}:=(\widetilde{\Ch}_{g_i})^0(E)=e^{2\pi i \frac{f_{i1}}{r}}c_1(E_{g_i}^1)+e^{2\pi i \frac{f_{i2}}{r}}c_1(E_{g_i}^2).
\end{equation}

 We introduce variables $q$ to keep track of the second Chern class $c_2(E)$ of the torsion free sheaf $E$, 
$q_1, \cdots, q_{r-1}$ to keep track of the classes $n_i$ for $i=1,\cdots, r-1$. 
Then we write 
\begin{align*}
&(-2)^{-\dim}\cdot (-2)^{1-2g}(-1)^{n}\cdot \sum_{n=0}^{\infty}q^n\cdot q_1^{n_1}\cdots q_{r-1}^{n_{r-1}} \cdot \int_{\sC^{[n]}}\frac{1}{N^{\vir}}\\
&=\sum_{n=0}^{\infty}q^n \cdot q_1^{n_1}\cdots q_{r-1}^{n_{r-1}} \cdot \int_{\sC^{[n]}}(r^n\omega-2)^{n+1-2g}\cdot \frac{(1+r^n\omega)^{n-g}}{(1-r^n\omega)^{n+g}}\cdot (1-2r^n\omega)^g. 
\end{align*}
\begin{rmk}
Since the moduli of stable Higgs pairs on $\sS$ is isomorphic to the moduli space of  parabolic  Higgs pairs on 
$(S,C)$, see \cite{Jiang2}. We will see that the variables $q_1, \cdots, q_{r-1}$ will keep track of the parabolic degree of the sheaf $E$ on the curve $C\subset S$. 
\end{rmk}

For simplicity, we deal with the case $q_1=\cdots=q_{r-1}=1$. 
So 
\begin{align}\label{eqn_key_formula3}
&(-2)^{-\dim}\cdot (-2)^{1-2g}(-1)^{n}\cdot \sum_{n=0}^{\infty}q^n \int_{\sC^{[n]}}\frac{1}{N^{\vir}}\\
&=\sum_{n=0}^{\infty}q^n \int_{\sC^{[n]}}(r^n\omega-2)^{n+1-2g}\cdot \frac{(1+r^n\omega)^{n-g}}{(1-r^n\omega)^{n+g}}\cdot (1-2r^n\omega)^g. \nonumber
\end{align}

Since $\sC^{[n]}$ has dimension $n$,  the integrand in (\ref{eqn_key_formula3}) only involves the power of $\omega$, and 
$$\int_{\sC^{[n]}}(r^n\omega)^{n}=\int_{C^{[n]}}(\omega)^{n}=1,$$ therefore 
\begin{align}\label{eqn_key_formula4}
&(-2)^{-\dim}\cdot (-2)^{1-2g}(-1)^{n}\cdot \sum_{n=0}^{\infty}q^n \int_{\sC^{[n]}}\frac{1}{N^{\vir}}\\
&=\sum_{n=0}^{\infty}q^n \int_{C^{[n]}}(\omega-2)^{n+1-2g}\cdot \frac{(1+\omega)^{n-g}}{(1-\omega)^{n+g}}\cdot (1-2\omega)^g. \nonumber
\end{align}

Then we perform the same careful Contour integral calculations as in \cite[\S 8.5]{TT1} by using \cite[\S 6.3]{St2}. 
We have the following result:

\begin{thm}\label{thm_generating_function_root_quintic}
We have 
\begin{equation}\label{eqn_prop_formula}
\sum_{n=0}^{\infty}q^n \int_{\sC^{[n]}}\frac{1}{N^{\vir}}=
A\cdot (1-q)^{g-1}\left(1+\frac{1-3q}{\sqrt{(1-q)(1-9q)}}\right)^{1-g},
\end{equation}
where 
$A:=(-2)^{\dim}\cdot (-2)^{2g-1}$.    $\square$
\end{thm}

\subsection{Quintic surfaces with ADE singularities}\label{subsec_quintic_ADE}

Let $\sS$ be a quintic surface with isolated ADE singularities.  We take $\sS$ as a surface DM stack.  From \cite{Horikava}, the coarse moduli space of the DM stack $\sS$ lies in the 
component of smooth quintic surfaces in the moduli space of general type surfaces with topological invariants in 
(\ref{eqn_topology_invariants_quintic}).  This means that there exists a deformation family such that the smooth quintic surfaces can be deformed to quintic surfaces with ADE singularities.

Let us fix a quintic surface $\sS$, with $P_1,\cdots, P_s\in \sS$ the isolated singular points with ADE type. 
Let $G_1,\cdots, G_s$ be the local ADE finite group in $SU(2)$ corresponding to $P_1,\cdots, P_s\in \sS$. 
We use the notation $|G_i|$ to represent the set of conjugacy classes for $G_i$. 
Let $\sX=\mbox{Tot}(K_{\sS})$ be the total space of $K_{\sS}$, which is a Calabi-Yau smooth DM stack.  Choose a generating sheaf $\Xi$ on $\sS$ such that it contains all the irreducible representations of the local group $G_i$ of $P_i$.  
Fix a $K$-group class $\mathbf{c}\in K_0(\sS)$ (determining a Hilbert polynomial $H$),  and let $\N$ be the moduli space of stable Higgs pairs with $K$-group class $\mathbf{c}$. We work on the Vafa-Witten invariants 
$\VW$ for the moduli space $\N^{\perp}_{\mathbf{c}}$ of stable fixed determinant $K_{\sS}$ and 
trace-free  Higgs pairs with $K$-group class $\mathbf{c}$. 

In this case we have a similar result as in Proposition \ref{prop_second_fixed_loci_Hilbert_scheme}. In the first case that 
$\sZ_1=\emptyset$, we get the same result by deformation invariance of the Vafa-Witten invariants $\VW$. 
We include the calculation for another cases here. 

We calculate one vertical term as in \cite{TT1}, and explain this time it will not give the same invariants as in the smooth case. 
This case is that $[\sZ_1]=[\sZ_0]\in L_0(\sS)$ component in $\rM^{(2)}$. 
So in this case $\Phi: \sI_0\to \sI_1$ is an isomorphism. Therefore
$$\Hilb^{\mathbf{c}_0, \mathbf{c}_0}(\sS)=\Hilb^{\mathbf{c}_0}(\sS)$$
and 
$$E=\sI_{\sZ}\otimes K_{\sS}\oplus \sI_{\sZ}\cdot \Tt^{-1}; \quad
\phi=\mat{cc}0&0\\
1&0\rix : E\to E\otimes K_{\sS}\cdot \Tt,$$
where $\sZ\subset \sS$ is a zero dimensional substack with $K$-group class $\mathbf{c}_0$. 
We use the same arguments as in \cite[\S 8.7]{TT1} for the torsion sheaf 
$\sE_{\phi}$ on $\sX$, which is the twist 
\begin{equation}\label{eqn_FZ}
\sF_{\sZ}:=(\pi^*\sI_{\sZ}\otimes \sO_{2\sS})
\end{equation}
by $\pi^*K_{\sS}$. Look at the following exact sequence:
$$0\to \pi^*\sI_{\sZ}(-2\sS)\longrightarrow \pi^*\sI_{\sZ}\longrightarrow \sF_{\sZ}\to 0,$$
we have
$$R\Hom(\sF_{\sZ},\sF_{\sZ})\to R\Hom(\pi^*\sI_{\sZ}, \sF_{\sZ})\to 
R\Hom(\pi^*\sI_{\sZ}, \pi_*\sF_{\sZ}(2\sS)).$$
The second arrow is zero since the section $\sO(2\sS)$ cutting out $2\sS\subset \sX$ annihilates $\sF_{\sZ}$. So by adjunction and the formula $\pi_*\sF_{\sZ}=\sI_{\sZ}\oplus \sI_{\sZ}\otimes K_{\sS}^{-1}\cdot \Tt^{-1}$, we have 
\begin{multline*}
R\Hom(\sF_{\sZ},\sF_{\sZ})\cong R\Hom_{\sS}(\sI_{\sZ},\sI_{\sZ})\oplus R\Hom_{\sS}(\sI_{\sZ},\sI_{\sZ}\otimes K_{\sS}^{-1})\Tt^{-1} \\
\oplus  R\Hom_{\sS}(\sI_{\sZ},\sI_{\sZ}\otimes K^2_{\sS})\Tt^{2}[-1]\oplus  R\Hom_{\sS}(\sI_{\sZ},\sI_{\sZ}\otimes K_{\sS})\Tt[-1].
\end{multline*}
We calculate the perfect obstruction theory
$$R\Hom_{\sX}(\sF_{\sZ},\sF_{\sZ})_{\perp}[1]$$
which comes from taking trace-free parts of the first and last terms and we have $\Hom_{\perp}=\Ext^3_{\perp}=0$. 
We have:
\begin{multline}\label{eqn_FZ_deformation}
\Ext^1_{\sX}(\sF_{\sZ},\sF_{\sZ})_{\perp}= \Ext^1_{\sS}(\sI_{\sZ},\sI_{\sZ})\oplus \Ext^1_{\sS}(\sI_{\sZ},\sI_{\sZ}\otimes K_{\sS}^{-1})\Tt^{-1} \\
\oplus  \Hom_{\sS}(\sI_{\sZ},\sI_{\sZ}\otimes K^2_{\sS})\Tt^{2}\oplus  \Hom_{\sS}(\sI_{\sZ},\sI_{\sZ}\otimes K_{\sS})_0\Tt.
\end{multline}
The obstruction $\Ext^2_{\perp}$ is just the dual of the above tensored with $\Tt^{-1}$.  The first term in (\ref{eqn_FZ_deformation}) is $T_{\sZ}\Hilb^{\mathbf{c}_0}(\sS)$, the fixed part of the deformations. The last term 
$\Hom_{\sS}(\sI_{\sZ},\sI_{\sZ}\otimes K_{\sS})_0\Tt=0$ since $\sI_{\sZ}\to \sI_{\sZ}$ is an isomorphism and taking trace-free we get zero. 
So the fixed part of the obstruction vanishes by duality. This tells us that 
$$[\Hilb^{\mathbf{c}_0}(\sS)]^{\vir}=[\Hilb^{\mathbf{c}_0}(\sS)].$$
Then the virtual normal bundle is: 
\begin{multline}\label{eqn_FZ_virtual_normal_bundle}
N^{\vir}= \Big[\Ext^1_{\sS}(\sI_{\sZ},\sI_{\sZ}\otimes K_{\sS}^{-1})\Tt^{-1} \oplus  \Hom_{\sS}(\sI_{\sZ},\sI_{\sZ}\otimes K^2_{\sS})\Tt^{2}\Big]-\\
\Big[\Ext^1_{\sS}(\sI_{\sZ},\sI_{\sZ}\otimes K_{\sS})\Tt\oplus \Ext^1_{\sS}(\sI_{\sZ},\sI_{\sZ}\otimes K_{\sS}^2)\Tt^2\oplus   \Ext^2_{\sS}(\sI_{\sZ},\sI_{\sZ}\otimes K^{-1}_{\sS})\Tt^{-1}\Big].
\end{multline}

It is quite complicated to integrate to equivariant Chern class on $\Hilb^{\mathbf{c}_0}(\sS)$, but we can do an easy case. 
Let $(\sS, P)$ be a quintic surface with only a singular point $P\in \sS$ with $A_1$-type singularity. 
By choosing a constant modified Hilbert polynomial $1$ on $\sS$ such that under the orbifold Chern character morphism:
$$K_0(\sS)\to H^*(I\sS)=H^*(\sS)\oplus H^*(B\mu_2)$$
the class 
$1\mapsto (1,1)$
where the second $1$ means the trivial one dimensional $\mu_2$-representation.  Then in this case 
the Hilbert scheme 
$\Hilb^{1}(\sS)=\pt$
which is a point.  This can be seen as follows. Around the singular point $P$, there is an open affine neighborhood 
$P\in U\subset \sS$ such that 
$$U\cong [\cc^2/\mu_2]$$
where $\zeta\in \mu_2$ acts on $\cc^2$ by 
$$\zeta\cdot (x,y)=(\zeta x, \zeta^{-1}y).$$
The Hilbert scheme of one point on $P\in U$ corresponds to invariant $\mu_2$-representation of length $1$, which must be trivial.  Then integration in this case must be: 
\begin{align*}
\int_{\Hilb^1(\sS)}\frac{1}{e(N^{\vir})}=\int_{\pt}1=1.
\end{align*}

Here is  a degree two calculation. 
Let $2$ be a constant modified Hilbert polynomial  on $\sS$ such that under the orbifold Chern character morphism:
$$K_0(\sS)\to H^*(I\sS)=H^*(\sS)\oplus H^*(B\mu_2)$$
the class 
$2\mapsto (2,2)$
where the second $2$ means the regular two dimensional $\mu_2$-representation.  Then in this case 
the Hilbert scheme 
$\Hilb^{2}(\sS)=\widetilde{S}$
where 
$\sigma: \widetilde{S}\to S$
is the crepant resolution of the coarse moduli space $S$ of $\sS$.  Then $\widetilde{S}$ is still a smooth surface. 
Then the  integration in this case can be written as: 
\begin{align*}
\int_{\widetilde{S}}\frac{1}{e(N^{\vir})}&=\int_{\widetilde{S}}
\frac{e(T^*_{\widetilde{S}}\Tt)\cdot e(T_{\widetilde{S}}\otimes K_{\sS}^2\Tt^2)\cdot e(H^0(K_{\sS}^2)^*\Tt^{-1})}
{e(T_{\widetilde{S}}\otimes K_{\sS}^{-1}\Tt^{-1})\cdot e(H^0(K_{\sS}^2)\Tt^2)}\\
&=\int_{\widetilde{S}}\frac{t^2 c_{\frac{1}{t}}(T^*_{\widetilde{S}})\cdot (2t)^2 c_{\frac{1}{2t}}(T_{\widetilde{S}}\otimes K_{\sS}^2)\cdot (-t)^{\dim}}
{(-t)^2 c_{-\frac{1}{t}}(T_{\widetilde{S}}\otimes K_{\sS}^{-1})\cdot (2t)^{\dim}}
\end{align*}
Only the $t^0$ term contributes and we let $t=1$, and get 

$$(-2)^{\dim}\int_{\widetilde{S}}
\frac{(1-\widetilde{c}_1+\widetilde{c}_2)4(1+\frac{1}{2}c_1(T_{\widetilde{S}}\otimes K_{\sS}^2)+\frac{1}{4}c_2(T_{\widetilde{S}}\otimes K_{\sS}^2))}
{1-c_1(T_{\widetilde{S}}\otimes K_{\sS}^{-1})+c_2(T_{\widetilde{S}}\otimes K_{\sS}^{-1})}$$
where $\widetilde{c}_i=c_i(\widetilde{S})$.  
Let 
$c_1:=c_1(T_{\sS}),$
then $c_1(K_{\sS})=-c_1$. 
We use the same formula for the Chern classes:
$$c_1(T_{\widetilde{S}}\otimes K_{\sS}^{2})=\widetilde{c}_1-4c_1, \quad  c_2(T_{\widetilde{S}}\otimes K_{\sS}^{2})=\widetilde{c}_2-2\widetilde{c}_1\cdot c_1+4 c_1^2;$$ 
$$c_1(T_{\widetilde{S}}\otimes K_{\sS}^{-1})=\widetilde{c}_1+2c_1, \quad c_2(T_{\widetilde{S}}\otimes K_{\sS}^{-1})=\widetilde{c}_2+\widetilde{c}_1\cdot c_1+c_1^2 $$
and have 
\begin{align}\label{eqn_final_integral_ADE}
\int_{\widetilde{S}}\frac{1}{e(N^{\vir})}
&=(-2)^{-\dim}\int_{\widetilde{S}}\frac{(1-\widetilde{c}_1+\widetilde{c}_2)4(1+\frac{1}{2}(\widetilde{c}_1-4c_1)+\frac{1}{4}(\widetilde{c}_2-2\widetilde{c}_1\cdot c_1+4c_1^2))}
{1-(\widetilde{c}_1+2c_1)+\widetilde{c}_2+ \widetilde{c}_1\cdot c_1+c_1^2} \\
&=(-2)^{-\dim}\int_{\widetilde{S}}
(\widetilde{c}_2+14 \widetilde{c}_1\cdot c_1+ 4(\widetilde{c}_1)^2).  \nonumber
\end{align}

\subsubsection{Comparison with the Euler characteristic of the Hilbert scheme of points}

The generating function of the Hilbert schemes of points on $\sS$ (a surface with finite ADE singularities $P_1, \cdots, P_s$) has been studied in \cite{GNS}, \cite{Toda}.  We recall the formula for the surface $\sS$ with $A_n$ singularities from \cite{Toda}.  
Let $P_1,\cdots, P_s$ have singularity type $A_{n_1}, \cdots, A_{n_s}$. 
Let $\sS\to S$ be the map to its coarse moduli space and $\sigma: \widetilde{S}\to S$ be the minimal resolution. 
Toda used wall crossing formula to calculate that 
\begin{equation}\label{eqn_generating_Hilbert_scheme}
\sum_{n\geq 0}\chi(\Hilb^n(\sS))q^{n-\frac{\chi(\widetilde{S})}{24}}=\eta(q)^{-\chi(\widetilde{S})}\cdot \prod_{i=1}^{s}\Theta_{n_i}(q),
\end{equation}
where $\eta(q)=q^{\frac{1}{24}}\prod_{n\geq 1}(1-q^n)$ is the Dedekind eta function,  and 
$$
\Theta_{n}(q)=\sum_{(k_1,\cdots, k_n)\in\zz^n}q^{\sum_{1\leq i\leq j\leq n}k_i k_j}
e^{\frac{2\pi \sqrt{-1}}{n+2}(k_1+2k_2+\cdots+nk_n)}$$

The series $\Theta_n(q)$ 
is a $\qq$-linear combination of the theta series determined by some integer valued positive definite quadratic forms on $\zz^n$ and   $\Theta_n(q)$ is a modular form of weight $n/2$.  So 
the generating series (\ref{eqn_generating_Hilbert_scheme}) is a Fourier development of a meromorphic modular form of weight $-\chi(S)/2$ for some congruence subgroup in $SL_2(\zz)$.
So this implies that it should be related to the S-duality conjecture for such surface DM stacks $\sS$. 
But the Euler characteristic of $\Hilb^n(\sS)$ is not the same as the contribution of it to the Vafa-Witten invariants $\VW(\sS)$, which is the integration over its virtual fundamental cycle. 

If the surface DM stack $\sS$ is a smooth projective surface $S$, the formula (\ref{eqn_generating_Hilbert_scheme}) is reduced to the G\"ottsche formula
$$\sum_{n\geq 0}\chi(\Hilb^n(S))q^{n-\frac{\chi(S)}{24}}=\eta(q)^{-\chi(S)}.
$$

\section{Evidence of the proposal for the S-duality}\label{sec_S_duality_P2}

In this section we check the S-duality for the projective plane $\pp^2$ in rank two, which is proved in \cite{Jiang_twist}; and give predictions on other cases. 
We include an extra comparison with $\pp(1,2,2)$, a $2$-th root stack of $\pp^2$ which is not included in \cite{Jiang_twist}. 

The $\pp^2$ case is already discussed in \cite[\S 4.2]{VW}, where they use the mathematical result of Klyachko and Yoshioka.  We will see that the theory of counting invariants for the gauge group $SU(2)/\zz_2=SO(3)$ are given by the invariants for the moduli space of semistable torsion free sheaves on the $\mu_2$-gerbes on $\pp^2$.
And the $2$-th  root stacks on $\pp^2$ also gives the formula predicted by the S-duality. 

\subsection{S-duality for $\pp^2$}

\subsubsection{The partition function for $\pp^2$.}
We consider the projective plane $\pp^2$.  Let $M_{\pp^2}(2, c_1, \chi)$ be the moduli space of stable torsion free sheaves
of rank $2$, first Chern class $c_1$ and second Chern class $\chi$.  Since $K_{\pp^2}<0$, any semistable Higgs sheaf $(E,\phi)$ will have $\phi=0$. Therefore the moduli space of stable Higgs sheaves $\N_{\pp^2}(2, c_1, \chi)$ is isomorphic to $M_{\pp^2}(2, c_1, \chi)$. Also the space $M_{\pp^2}(2, c_1, \chi)$ is smooth and the Vafa-Witten invariants defined in (\ref{eqn_localized_SU_invariants_intr}) is just the Euler characteristic (up to a sign) of the moduli space.  Then we introduce 
\begin{equation}\label{eqn_partition_P2}
\widehat{Z}_{c_1}^{\pp^2}(q)=\sum_{\chi}e(M_{\pp^2}(2, c_1, \chi))q^{\chi}
\end{equation}
Let $N_{\pp^2}(2, c_1, \chi)$ be the moduli space of stable vector bundles 
of rank $2$, first Chern class $c_1$ and second Chern class $\chi$.  Let 
$$Z_{c_1}^{\vb, \pp^2}(q)=\sum_{\chi}e(N_{\pp^2}(2, c_1, \chi))q^{\chi}$$
be the partition function. 
Then from \cite{Yoshioka}, \cite{Kool}, 
$$\widehat{Z}_{c_1}^{\pp^2}(q)=\frac{q^{\frac{1}{8}}}{\eta(q)^{\chi(\pp^2)}}\cdot Z_{c_1}^{\vb, \pp^2}(q)$$
where $\eta(q)$ is the Dedekind eta function. 

To state the result we introduce some notations. First let $H(\Delta)$ be the Hurwitz class numbers, i.e., 
$H(\Delta)$ is the number of positive definite integer 
 binary quadratic forms $AX^2+BXY+CY^2$ such that $B^2-4AC=-\Delta$ and weighted by the size of its automorphisms group. Let 
$\sigma_0(n)$ be the divisor function. 

\begin{thm}\label{thm_partition_P2}(\cite{Klyachko}, \cite{Yoshioka}, \cite{Kool})
We have:
$$
Z_{c_1}^{\vb, \pp^2}(q)=\begin{cases}
q^{\frac{1}{4}c_1^2+\frac{3}{2}c_1+2}\cdot \sum_{n=1}^{\infty}3 H(4n-1)q^{\frac{1}{4}-n};  & (c_1 \text{~odd});\\
q^{\frac{1}{4}c_1^2+\frac{3}{2}c_1+2}\cdot \sum_{n=1}^{\infty}3 \left(H(4n)-\frac{1}{2}\sigma_0(n)\right)q^{-n};  & (c_1 \text{~even}).
\end{cases}
$$
\end{thm}
In the case $c_1$ is odd, by the work of D. Zagier \cite{Zagier} we see that $Z_{c_1}^{\vb, \pp^2}(q)$ is the holomorphic part of a modular form of weight $3/2$ for $\Gamma_0(4)$ (up to
replacing $q$ by $q^{-1}$ and up to an overall power of $q$ in front). In the case $c_1$ is even one
only obtains modularity after correctly adding strictly semistable sheaves to the moduli
space. Their contribution turns out to cancel the sum of divisors term. 

In order to compare with other partition functions later we introduce:
\begin{equation}\label{eqn_partition_odd_P2}
Z_{c_1, \odd}^{\vb, \pp^2}(q):=q^{\frac{1}{4}c_1^2+\frac{3}{2}c_1+2}\cdot \sum_{n=1}^{\infty}3 H(4n-1)q^{\frac{1}{4}-n};
\end{equation}
and 
\begin{equation}\label{eqn_partition_even_P2}
Z_{c_1, \even}^{\vb, \pp^2}(q):=q^{\frac{1}{4}c_1^2+\frac{3}{2}c_1+2}\cdot \sum_{n=1}^{\infty}3 \left(H(4n)-\frac{1}{2}\sigma_0(n)\right)q^{-n}.
\end{equation}

\subsubsection{The partition function for $\pp(1,2,2)$ and $\pp(2,2,2)$.}
In \cite{GJK}, the authors generalize the calculation of the moduli of stable torsion free sheaves on smooth toric variety to weighted projective spaces $\pp(a, b, c)$, which is  a special toric DM stack. 
The calculation uses and generalizes  the toric method in \cite{Kool} to this toric DM stack. 
We omit the detail calculation and only 
 include the calculation results for 
$\pp(1,2,2)$ and $\pp(2,2,2)$.

In general the weighted projective plane $\pp(a,b,c)$ is a $\mu_d$-gerbe $\pp(\frac{a}{d},\frac{b}{d},\frac{c}{d})$ where 
$d=\gcd(a,b,c)$.  In the case of $\pp(2,2,2)$ which is a $\mu_2$-gerbe over $\pp^2$,  the partition function will depend on the choice of the component of the inertia stack $I\pp(2,2,2)=\pp(2,2,2)\cup \pp(2,2,2)$. We use $\lambda=0, \text{or~} 1$ to distinct these two components. 

Let $N_{\pp(1,2,2)}(2, c_1, \chi)$ be the moduli space of stable vector bundles 
of rank $2$, first Chern class $c_1$ and second Chern class $\chi$.  Let 
$$Z_{c_1}^{\vb, \pp(1,2,2)}(q)=\sum_{\chi}e(N_{\pp(1,2,2)}(2, c_1, \chi))q^{\chi}$$
be the partition function. 
The results in \cite[Theorem 1.2]{GJK} is stated as follows:
\begin{thm}\label{thm_partition_P122}(\cite[Theorem 1.2]{GJK})
We have:
\begin{align*}
&Z_{c_1}^{\vb, \pp(1,2,2)}(q)=\\
&\begin{cases}
q^{\frac{1}{8}c_1^2+\frac{3}{2}c_1+\frac{17}{4}}\cdot \sum_{n=1}^{\infty}H(8n-1)q^{\frac{1}{8}-n};  & (c_1 \text{~odd});\\
q^{\frac{1}{8}c_1^2+\frac{3}{2}c_1+4}\cdot \left(q^{\frac{1}{2}}\sum_{n=1}^{\infty}3H(4n-1)q^{\frac{1}{2}-2n}+\sum_{n=1}^{\infty}3(H(4n)-\frac{1}{2}\sigma_0(n))q^{-2n}\right);  & (c_1\equiv 0  \text{~mod~} 4);\\
q^{\frac{1}{8}c_1^2+\frac{3}{2}c_1+4}\cdot \left(\sum_{n=1}^{\infty}3H(4n-1)q^{\frac{1}{2}-2n}+q^{\frac{1}{2}}\sum_{n=1}^{\infty}3(H(4n)-\frac{1}{2}\sigma_0(n))q^{-2n}\right);  & (c_1\equiv 2  \text{~mod~} 4)
\end{cases}
\end{align*}
\end{thm}

\begin{prop}\label{prop_P122_P2}
Let 
$$Z_{c_1,c_1\equiv 0  \text{~mod~} 4}^{\vb, \pp(1,2,2)}(q)=q^{\frac{1}{8}c_1^2+\frac{3}{2}c_1+4}\cdot \left(q^{\frac{1}{2}}\sum_{n=1}^{\infty}3H(4n-1)q^{\frac{1}{2}-2n}+\sum_{n=1}^{\infty}3(H(4n)-\frac{1}{2}\sigma_0(n))q^{-2n}\right)$$
Then we have:
$$Z_{c_1,c_1\equiv 0  \text{~mod~} 4}^{\vb, \pp(1,2,2)}(q^{\frac{1}{2}})=
q^{-\frac{3}{16}c_1^2-\frac{3}{4}c_1+\frac{1}{4}}\cdot Z_{c_1,\odd}^{\vb, \pp^2}(q)
+q^{-\frac{3}{16}c_1^2-\frac{3}{4}c_1}\cdot Z_{c_1,\even}^{\vb, \pp^2}(q).$$
\end{prop}
\begin{proof}
From Theorem \ref{thm_partition_P122}, we calculate:
\begin{align*}
&Z_{c_1,c_1\equiv 0  \text{~mod~} 4}^{\vb, \pp(1,2,2)}(q^{\frac{1}{2}})\\
&=q^{\frac{1}{16}c_1^2+\frac{3}{4}c_1+2}\cdot \left(q^{\frac{1}{4}}\sum_{n=1}^{\infty}3H(4n-1)q^{\frac{1}{4}-n}+\sum_{n=1}^{\infty}3(H(4n)-\frac{1}{2}\sigma_0(n))q^{-n}\right)\\
&=q^{-\frac{3}{16}c_1^2-\frac{3}{4}c_1+\frac{1}{4}}\cdot \left(q^{\frac{1}{4}c_1^2+\frac{3}{2}c_1+2}\cdot \sum_{n=1}^{\infty}3H(4n-1)q^{\frac{1}{4}-n}\right)+ q^{-\frac{3}{16}c_1^2-\frac{3}{4}c_1+\frac{1}{4}}\cdot\sum_{n=1}^{\infty}3(H(4n)-\frac{1}{2}\sigma_0(n))q^{-n}\\
&=q^{-\frac{3}{16}c_1^2-\frac{3}{4}c_1+\frac{1}{4}}\cdot Z_{c_1,\odd}^{\vb, \pp^2}(q)
+q^{-\frac{3}{16}c_1^2-\frac{3}{4}c_1}\cdot Z_{c_1,\even}^{\vb, \pp^2}(q).
\end{align*}
\end{proof}

Now we list the results for $\pp(2,2,2)$. In this case the first Chern class $c_1$ is always even. 
Let $\lambda\in\{0,1\}$ index the component in the inertia stack $I\pp(2,2,2)=\pp(2,2,2)\cup\pp(2,2,2)$.
Let  $N_{\pp(2,2,2)}(2, c_1, \chi)$ be the moduli space of stable vector bundles
of rank $2$, first Chern class $c_1$ and second Chern class $\chi$.  Let 
$$Z_{c_1,\lambda}^{\vb, \pp(2,2,2)}(q)=\sum_{\chi}e(N_{\pp(2,2,2)}(2, c_1, \chi))q^{\chi}$$
be the partition function. 

\begin{thm}\label{thm_partition_P222}(\cite[Theorem 1.2]{GJK})
Since $c_1$ is even, there are two cases $c_1\equiv 0 (\text{mod~} 4)$ or $c_1\equiv 2 (\text{mod~} 4)$.
We have 
\begin{align*}
&Z_{c_1, 0}^{\vb, \pp(2,2,2)}(q)=\\
&\begin{cases}
Z_{\frac{c_1}{2}}^{\vb, \pp^2}(q)=q^{\frac{1}{16}c_1^2+\frac{3}{4}c_1+2}\cdot \sum_{n=1}^{\infty}3(H(4n)-\frac{1}{2}\sigma_0(n))q^{-n};  & (c_1\equiv 0  \text{~mod~} 4);\\
Z_{\frac{c_1}{2}}^{\vb, \pp^2}(q)=q^{\frac{1}{16}c_1^2+\frac{3}{4}c_1+2}\cdot \sum_{n=1}^{\infty}3(H(4n-1)q^{\frac{1}{4}-n};  & (c_1\equiv 2  \text{~mod~} 4)
\end{cases}
\end{align*}
and 
\begin{align*}
&Z_{c_1, 1}^{\vb, \pp(2,2,2)}(q)=\\
&\begin{cases}
Z_{\frac{c_1}{2}+1}^{\vb, \pp^2}(q)=q^{\frac{1}{4}(\frac{c_1}{2}+1)^2+\frac{3}{2}(\frac{c_1}{2}+1)+2}\cdot \sum_{n=1}^{\infty}
3(H(4n-1)q^{\frac{1}{4}-n};  & (c_1\equiv 0  \text{~mod~} 4);\\
Z_{\frac{c_1}{2}+1}^{\vb, \pp^2}(q)=q^{\frac{1}{4}(\frac{c_1}{2}+1)^2+\frac{3}{2}(\frac{c_1}{2}+1)+2}\cdot  \sum_{n=1}^{\infty}3(H(4n)-\frac{1}{2}\sigma_0(n))q^{-n};  & (c_1\equiv 2  \text{~mod~} 4).
\end{cases}
\end{align*}
\end{thm}
\begin{rmk}
In the case $c_1$ is even,  from \cite{GJK} one only obtains modularity after correctly adding strictly semistable sheaves to the moduli space. Their contribution turns out to cancel the sum of divisors term.  Thus in the following when checking the S-duality, we can ignore the divisor functions. 
\end{rmk}

\subsubsection{S-duality}

From (4.30) of \cite[\S 4]{VW},  by a result of Zagier \cite{Zagier}, let 
$$f_0=\sum_{n\geq 0}3H(4n)q^n+6\tau_2^{-\frac{1}{2}}\sum_{n\in \zz}\beta(4\pi n^2\tau_2)q^{-n^2}$$
and
$$f_1=\sum_{n>0}3H(4n-1)q^{n-\frac{1}{4}}+6\tau_2^{-\frac{1}{2}}\sum_{n\in \zz}\beta(4\pi (n+\frac{1}{2})^2\tau_2)q^{-(n+\frac{1}{2})^2}$$
where $q^{2\pi i \tau}$, and $\tau_2=\Im(\tau)$, and 
$$\beta(t)=\frac{1}{16\pi}\int_{1}^{\infty}u^{-\frac{3}{2}}e^{-ut}du.$$
From \cite{Zagier},  these functions are modular, but not holomorphic.  Hence $Z_{c_1}^{\vb, \pp^2}$ ($c_1$ even or odd) is the homomorphic part of the non-holomorphic modular functions above.  Also by Zagier, see \cite[Formula (4.31)]{VW}, under $\tau\mapsto -\frac{1}{\tau}$, 
we have:
\begin{equation}\label{eqn_S_transformation_P2}
\mat{c} f_0(-\frac{1}{\tau})\\
f_1(-\frac{1}{\tau})\rix=
(\frac{\tau}{i})^{\frac{3}{2}}\cdot \left(-\frac{1}{\sqrt{2}}\right)
\mat{cc} 1&1\\
1&-1\rix\mat{c} f_0(\tau)\\
f_1(\tau)\rix.
\end{equation}
This is the transformation conjecture (\ref{eqn_S_transformation}).  We know that $f_0$ is invariant under $T$ and $f_1$ is invariant under $T^4$. Therefore $f_0$ is invariant under $ST^4S$. 

To check the S-duality, we choose the case $c_1=0$ or $2$, from Theorem \ref{thm_partition_P222} and Theorem \ref{thm_partition_P2} we calculate:
$$
\begin{cases}
Z_{0,0}^{\vb, \pp(2,2,2)}(q)=Z_{0}^{\vb,\pp^2}(q)=q^2\sum_{n=1}^{\infty}3(H(4n)-\frac{1}{2}\sigma_0(n))q^{-n}; \\
Z_{2,0}^{\vb, \pp(2,2,2)}(q)=Z_{1}^{\vb,\pp^2}(q)= q^{\frac{15}{4}}\cdot \sum_{n=1}^{\infty}3(H(4n-1)q^{\frac{1}{4}-n};\\
Z_{0,1}^{\vb, \pp(2,2,2)}(q)=Z_{1}^{\vb,\pp^2}(q)= q^{\frac{15}{4}}\cdot \sum_{n=1}^{\infty}3(H(4n-1)q^{\frac{1}{4}-n};\\
Z_{2,1}^{\vb, \pp(2,2,2)}(q)=Z_{2}^{\vb,\pp^2}(q)=q^6\sum_{n=1}^{\infty}3(H(4n)-\frac{1}{2}\sigma_0(n))q^{-n}.
\end{cases}
$$

Then we check that under transformation $\tau\mapsto -\frac{1}{\tau}$, we have
\begin{equation}\label{eqn_S_P2-P222_1}
q^{-2}\cdot Z_{0}^{\vb,\pp^2}(q)\mapsto 
q^{-2}Z_{0,0}^{\vb, \pp(2,2,2)}(q)+ q^{-\frac{15}{4}}\cdot Z_{0,1}^{\vb, \pp(2,2,2)}(q).
\end{equation}
and 
\begin{equation}\label{eqn_S_P2-P222_2}
q^{-\frac{15}{4}}\cdot Z_{1}^{\vb,\pp^2}(q)\mapsto 
q^{-6}Z_{2,1}^{\vb, \pp(2,2,2)}(q)- q^{-\frac{15}{4}}\cdot Z_{2,0}^{\vb, \pp(2,2,2)}(q).
\end{equation}
We define
\begin{defn}\label{defn_partition_SO3}
We define 
$$Z^{\pp^2}_0(\tau, SU(2)/\zz_2):=\frac{1}{2}\cdot\left(q^{-2}\cdot Z_{0,0}^{\vb, \pp(2,2,2)}(q)+q^{-\frac{15}{4}}\cdot Z_{0,1}^{\vb, \pp(2,2,2)}(q)\right)$$
\end{defn}
Then from the above calculations in (\ref{eqn_S_P2-P222_1}), we have
\begin{thm}\label{thm_S-duality_P2}(\cite{Jiang_twist})
Write 
$$Z_0^{\pp^2}\left(\tau, SU(2)\right)=q^{-2}\cdot Z_{0}^{\vb,\pp^2}(q).$$
Under the $S$-transformation  $\tau\mapsto -\frac{1}{\tau}$, we have:
$$
Z_0^{\pp^2}\left(-\frac{1}{\tau}, SU(2)\right)=\pm 2^{-\frac{3}{2}}\left(\frac{\tau}{i}\right)^{\frac{3}{2}}Z_0^{\pp^2}(\tau, SU(2)/\zz_2).
$$
\end{thm}
\begin{proof}
This is Calculation (\ref{eqn_S_P2-P222_1}).  Then the S-duality holds based on the observation that  the partition function $Z_{0,0}^{\vb, \pp(2,2,2)}(q)$ is the same as 
the partition fucntion $Z_{0}^{\vb, [\pp^2/\mu_2]}(q)$, where $[\pp^2/\mu_2]$ is the trivial $\mu_2$-gerbe on $\pp^2$. 
\end{proof}

\subsection{Discussion on other cases}

Let $S$ be a smooth projective surface.  From Theorem \ref{thm_S-duality_P2},  it is reasonable to conjecture that some
$\mu_r$-gerbe $\sS\to S$ is the candidate to define the Vafa-Witten invariants for the gauge group $SU(r)/\zz_r$. We provide a reasonable explanation.  Consider the exact sequence
$$1\to \zz_r\longrightarrow SU(r)\longrightarrow SU(r)/\zz_r\to 1,$$
and let $\N_{(r, c_1, \chi)}(\sS)$ be the moduli space of stable or semistable Higgs sheaves $(E,\phi)$ with topological data 
$(r,c_1,\chi)$.  We even conjecture that we should use the moduli space $\N^{\tiny\mbox{tw}}_{(r, c_1, \chi)}(\sS)$ of $\mu_r$-gerbe $\sS$ twisted Higgs sheaves on $\sS$.  

We review the gerbe twisted sheaves here following \cite{Lieblich_Duke}.
Let $\chi: \mu_r\to \cc^*$ be the character morphism. 
Let $E$ be a torsion free sheaf, there is a natural gerbe action $E\times \mu_r\to E$ on $E$.
From \cite{Lieblich_Duke}, given an $\sO_S$-module $E$, the module action $m: \cc^* \times E\to E$ yields an
associated right action $m^\prime: E\times \cc^*\to E$ with 
$m^\prime(s, \varphi)= m(\varphi^{-1}, s)$. This is
always called the associated right action.
\begin{defn}\label{defn_gerbe_twisted_sheaf}
A  $\mu_r$-gerbe $\sS$ twisted Higgs sheaf $(E,\phi)$ is given by a gerbe twisted  torsion free sheaf  $E$ which is given by
\[
\xymatrix{
E\times \mu_r \ar[r]\ar[d]_{\chi}& E\ar[d]^{\id}\ar[r]^{\phi}&E\otimes K_{\sS}\ar[d]^{\id}\\
 E\times \cc^* \ar[r]^{m}& E\ar[r]^{\phi}& E\otimes K_{\sS}.
}
\]
\end{defn}

This diagram is compatible with the Higgs field morphism $\phi$. 
Since $\mu_r$-action is trivial as a gerbe structure, this induces a $PGL_r$-Higgs sheaf $(E,\phi)$ on $S$. 
A similar picture for $\cc^*$-gerbe twisted sheaves can be found in  \cite{Yoshioka2}.
More details will provided elsewhere.  
Therefore it is promising to take the moduli space  $\N^{\tiny\mbox{tw}}_{(r, c_1, \chi)}(\sS)$  as the candidate to check the S-duality (\ref{eqn_S_transformation}).
We will make a general proposal for the S-duality conjecture in \cite{Jiang_twist} and prove for K3 surfaces. 

Since the Langlands dual group $^{L}SU(r)=SU(r)/\zz_r$, our original idea is to use the Vafa-Witten invariants for the global quotient stack $[S/\zz_r]$ to get the invarisnts for $SU(r)/\zz_r$.  It is still interesting to attack this prediction.  Another interesting case is the root stack $\sqrt[d]{(S,D)}$ of the smooth surface $S$ with a simple normal crossing divisor $D$.  The author is not aware if this root stack is also a candidate for the S-duality.  One fact is that in the case $\pp(1,2,2)$ which is a $2$-th root stack over $\pp^2$ with respect to the standard divisor $\pp^1\subset \pp^2$,  from Proposition  \ref{prop_P122_P2}, the partition function of $\pp(1,2,2)$ can also give the partition function of $\pp^2$ after the $S$-transformation $\tau\mapsto -\frac{1}{\tau}$.


\subsection*{}


\begin{thebibliography}{12}  
\bibitem{ACGH}E. Arbarello, M. Cornalba, P.A. Griffiths and J. Harris, \newblock {\em Geometry of algebraic curves}, Volume I, Springer-Verlag (1985). 
\bibitem{Behrend}  K. Behrend, \newblock Donaldson-Thomas invariants via microlocal geometry, 
{\em Ann. Math.} (2009), Vol. 170, No.3, 1307-1338, math.AG/0507523.
\bibitem{BF}   K. Behrend and B. Fantechi, \newblock The intrinsic normal cone,
                   alg-geom/9601010, {\em Invent. Math}. 128 (1997), no. 1, 45-88.
                     
\bibitem{GP}T.~Graber and R.~Pandharipande, \newblock Localization of virtual classes, {\em Invent.~Math.}~\textbf{135} 487-518, 1999. alg-geom/9708001.

\bibitem{GJK} A. Gholampour, Y. Jiang, and Martijn Kool, \newblock Sheaves on weighted projective planes and modular forms, {\em Advances in Theoretical and Mathematical Physics}, Vol. 21, NO. 6, (2017) 1455-1524, arXiv:1209.3922.
\bibitem{GT}A. Gholampour and R.P. Thomas, \newblock  Degeneracy loci, virtual cycles and nested Hilbert schemes, arXiv:1709.06105. 
\bibitem{Gottsche}L. G\"ottsche, \newblock 
Modular forms and Donaldson invaraints for 4-manifolds with $b_+=1$,  {\em Journal of the American Mathematical Society},   Vol. 9, No. 3 (1996), 827-843.

\bibitem{GK}L. G\"ottsche and M. Kool, \newblock Virtual refinements of the Vafa-Witten formula,
arXiv:1703.07196.
\bibitem{GNS} A. Gyenge,  A. N\'emethi and B. Szendroi, \newblock  Euler characteristics of Hilbert schemes of points on simple surface singularities, published online, European Journal of Mathematics. 

\bibitem{Horikava}E. Horikava, \newblock On deformations of quintic surfaces, {\em  Inventiones mathematicae}, 
(1975), Vol 31, Iss 1,  43-85.

 \bibitem{HL}D. Huybrechts and M. Lehn, \newblock {\em The geometry of moduli spaces of sheaves}, Aspects of Mathematics, E31, Friedr. Vieweg \& Sohn, Braunschweig, 1997. MR MR1450870 (98g:14012).


\bibitem{JT}Y. Jiang and R. Thomas, \newblock Virtual signed Euler characteristics, {\em Journal of Algebraic Geometry}, 26 (2017) 379-397, arXiv:1408.2541.
\bibitem{Jiang}Y. Jiang, \newblock Note on MacPherson's local Euler obstruction, {\em Michigan Mathematical Journal}, 68 (2019),
227-250, arXiv:1412.3720.
\bibitem{Jiang2}Y. Jiang,  \newblock The Tanaka-Thomas's Vafa-Witten invariants for surface DM stacks II: Root stacks and parabolic Higgs pairs, in preparation. 
\bibitem{JKT}Y. Jiang, P. Kundu and H.-H. Tseng, \newblock  The Bogomolov inequality for surface Deligne-Mumford stacks, in preparation. 
\bibitem{JP}Y. Jiang, and P. Kundu, \newblock  The Tanaka-Thomas's Vafa-Witten invariants for surface Deligne-Mumford stacks, arXiv:1903.11477.
\bibitem{Jiang_twist}Y. Jiang, \newblock Counting twisted sheaves and S-duality, preprint. 
\bibitem{JS}D. Joyce and Y. Song, \newblock A theory of generalized Donaldson-Thomas invariants, 
{\em Memoris of the AMS},  217 (2012),  1-216,  
arXiv:0810.5645. 

\bibitem{Joyce} D. Joyce, \newblock A classical model for derived critical locus,  {\em Journal of Differential Geometry}, 101 (2015), 289-367, arXiv:1304.4508.
\bibitem{Kapranov}M. Kapranov, \newblock The elliptic curve in the S-duality theory and Eisenstein series for Kac-Moody groups, arXiv:math/0001005v2 [math.AG]. 
\bibitem{KL}Y.-H. Kiem and J. Li, \newblock Localizing virtual cycles by cosections, {\em Jour. A.M.S.} 26
1025-1050, (2013),  arXiv:1007.3085.

\bibitem{Klyachko} A. A. Klyachko, \newblock  Vector bundles and torsion free sheaves on the projective plane,  preprint, Max  Planck Institut fur Mathematik (1991).

 \bibitem{KM}D. Kotschick and  J. W. Morgan, \newblock 
SO(3)-Invariants for 4-manifolds with $ b_2^+= 1$ II, {\em  J. Differential  Geometry},  39 (1994) 433-456.

\bibitem{KW}A. Kapustin and E. Witten, \newblock Electric-Magnetic Duality
And The Geometric Langlands Program, 	arXiv:hep-th/0604151.

\bibitem{Kool}M. Kool, \newblock Fixed point loci of moduli spaces of sheaves on toric varieties, {\em Adv. Math.}, 
Vol. 227, Iss. 4, (2011), 1700-1755.
\bibitem{Laarakker}T. Laarakker,  \newblock Monopole contributions to refined Vafa-Witten invariants, arXiv:1810.00385.
\bibitem{LMB}G. Laumon and L. Moret-Bailly, \newblock  {\em Champs Algebriques},   Ergebnisse der Mathematik und ihrer Grenzgebiete. 3. Folge / A Series of Modern Surveys in Mathematics.

\bibitem{LT} J. Li and G. Tian, \newblock  Virtual moduli cycles and Gromov-Witten 
invariants of algebraic varieties, {\em J. Amer. Math. Soc.}, 11, 119-174, 1998, math.AG/9602007. 
\bibitem{LQ}W. Li and Z. Qin, \newblock   On blowup formulae for the $S$-duality conjecture of Vafa and Witten.
{\em Invent. Math.} 136 (1999), 451-482. 

\bibitem{Lieblich_Duke}M . Lieblich, \newblock Moduli of twisted sheaves, {\em Duke Math. J.}, 
Vol. 138, No. 1 (2007), 23-118.
\bibitem{MT}D. Maulik and R. P. Thomas, \newblock Sheaf counting on local K3  surfaces,  arXiv:1806.02657.    
\bibitem{MO}C. Montonen and D. I. Olive, \newblock Magnetic Monopoles As Gauge Particles, {\em Phys. Lett.} B 72, 117 (1977).
\bibitem{MW} G. Moore and E. Witten, \newblock Integration over the u-plane in Donaldson theory,  {\em Adv. Theor. Math. Phys.} 1 (1997), no. 2, pages 298-387. 


\bibitem{Nironi}F.  Nironi, \newblock Moduli Spaces of Semistable Sheaves on Projective Deligne-Mumford Stacks, 
arXiv:0811.1949. 
\bibitem{OS03}M. Olsson and J. Starr, \newblock Quot functors for Deligne-Mumford stacks, {\em Comm. Algebra} 31 (2003), no. 8, 4069-4096, Special issue in honor of Steven L. Kleiman. 
\bibitem{Stack_Project} Stack Project: https://stacks.math.columbia.edu/download/stacks-morphisms.pdf.
\bibitem{St2}R. Stanley, \newblock {\em Enumerative conbinatorics Vol 2},  Cambridge University Press (1999).
\bibitem{TT1} Y. Tanaka and R. P. Thomas, \newblock Vafa-Witten invariants for projective surfaces I: stable case, 
{\em J. Alg. Geom.} to appear, 
arXiv.1702.08487.
\bibitem{TT2}Y. Tanaka and R. P. Thomas, \newblock Vafa-Witten invariants for projective surfaces II: semistable case, {\em Pure Appl. Math. Q.} 13  (2017), 517-562, Special Issue in Honor of Simon Donaldson, 
arXiv.1702.08488.
\bibitem{Thomas} R. P. Thomas, \newblock  A holomorphic Casson invariant for Calabi-Yau 3-folds,
                      and bundles on K3 fibrations, {\em J. Differential Geom.}, 54, 367-438, 2000.
                       math.AG/9806111.
\bibitem{Thomas2} R. P. Thomas, \newblock  Equivariant K-theory and refined Vafa-Witten invariants, preprint, 
arXiv:1810.00078.
                                            
                       
\bibitem{Toda} Y. Toda, \newblock   S-Duality for surfaces with   An -type singularities, {\em Mathematische Annalen} 
(2015), Vol. 363, Issue 1-2,  679-699. 
                 
\bibitem{Toda2}Y. Toda, \newblock    Curve counting theories via stable objects II: DT/ncDT flop formula, {J. reine angew. Math.}  675 (2013), 1-51.   
\bibitem{VW} C.~Vafa and E.~Witten, \emph{A strong coupling test of S-duality}, Nucl. Phys. \textbf{B\,431} 3--77, 1994. hep-th/9408074.
\bibitem{Yoshioka}K. Yoshioka, \newblock The Betti numbers of the moduli space of stable sheaves of rank 2 on $\pp^2$, Kyoto University. 
\bibitem{Yoshioka2}K. Yoshioka, \newblock Moduli spaces of twisted sheaves on a projective variety,  Advanced Studies in Pure Mathematics 45, 2006 Moduli Spaces and Arithmetic Geometry (Kyoto, 2004) pp. 1-42. 

\bibitem{Zagier}D. Zagier, \newblock Nombres de classes et formes modulaires de poids $3/2$, {\em C. R. Acad. Sc. Paris}, 281A, 
883-886 (1975).
\end{thebibliography}
\end{document}